\documentclass[12pt,a4paper]{amsart}

\usepackage{amsmath,amsfonts,amssymb, amsthm}
\usepackage{color}
\usepackage[hmargin=2cm,vmargin=2cm]{geometry}

\usepackage{hyperref}
\usepackage{nameref,zref-xr}                    
\usepackage[all]{xy}   
\usepackage{tikz}
\usepackage{tikz-cd}
\usepackage{quiver}

\usepackage{caption}
\usepackage{subcaption}
\usepackage{array}

\usepackage{enumerate}

\newtheorem{theorem}{Theorem}[section]     
 
\newtheorem{lemma}[theorem]{Lemma} 
\newtheorem{corollary}[theorem]{Corollary}

\theoremstyle{definition}

\theoremstyle{remark}
\newtheorem{remark}[theorem]{Remark}
\newtheorem{example}[theorem]{Example}
\newtheorem*{theorem*}{Theorem}

\newcommand{\mbA}{\mathbb A}
\newcommand{\mbZ}{\mathbb Z}
\newcommand{\mbQ}{\mathbb Q}
\newcommand{\mbC}{\mathbb C}
\newcommand{\mbR}{\mathbb R}

\newcommand{\pt}{\mathrm{pt}}

\newcommand{\dd}{\mathbf{d}}

\newcommand{\Tr}{\mathrm{Tr}}

\newcommand{\im}{\mathrm{im}}

\newcommand{\Hom}{\mathrm{Hom}}

\newcommand{\ad}{\mathrm{ad}}

\newcommand{\colim}{\mathrm{colim}}

\newcommand{\End}{\mathrm{End}}
\newcommand{\Supp}{\mathrm{Supp}}

\newcommand{\g}{\mathfrak{g}}
\newcommand{\p}{\mathfrak{p}}
\newcommand{\el}{\mathfrak{l}}
\newcommand{\U}{\mathfrak{u}}

\newcommand{\T}{\mathfrak{t}}
\newcommand{\gl}{\mathfrak{gl}}

\newcommand{\fr}{\mathrm fr}

\newcommand{\BM}{\mathrm{BM}}

\newcommand{\Crit}{\mathrm{Crit}}
\newcommand{\mbO}{\mathbb O}
\newcommand{\mbG}{\mathbb G}
\newcommand{\mbD}{\mathbb D}
\newcommand{\G}{\widehat{G}}

\newcommand{\mbN}{\mathbb N}

\begin{document}

\title{Hall induction for cotangent representations and wheel conditions}

\author{Danil Gubarevich}
\address[D. Gubarevich]{Université de Versailles St-Quentin, 45 Avenue des Etats Unis, 78000 Versailles, France}

\begin{abstract}
In this short note we study the Hall induction of cotangent representations of reductive groups. We prove torsion freeness of its equivariant Borel-Moore homology. In K-theory we find an analog of wheel conditions verified by the image of restriction map to the fixed point and consider examples.
\end{abstract}

\maketitle

\tableofcontents

\section{The results}

In this paper we consider the cotangent stack $T^*(V/G)$ of the quotient stack $V/G$ of a representation $V$ of complex reductive group $G$. We consider the equivariant Borel-Moore homology and equivariant K-theory of $T^*(V/G)\simeq \mu_V^{-1}(0)/G$, where $\mu_V:T^*V \to \g^*$ is the moment map. We adapt several proven results on cohomological Hall algebra(CoHA) of a quiver to this situation, where the role om Hall product is played by the Hall induction. To define the induction map we first use the dimensional reduction isomorphism to work with vanishing cycle cohomology and define the associative induction on the level of vanishing cycle cohomology. 

Our first statement is about torsion freeness of equivariant Borel-Moore homology. The statement involves a sufficiently big torus $T_s$ that is assumed to act on $\mu_V^{-1}(0)/G$. Suppose $T_s$ is subject to the following \textbf{assumptions}:

\begin{itemize}
     \item [$\cdot$] $T_s$ acts on $T^*V\times \g$, preserves $\mu_V^{-1}(0)\subset T^*V\times \g$, and commutes  with the action of $G$,
    \item [$\cdot$] the function $f:T^*V\times \g\to\mbC, ~(x,x^*,\xi)\mapsto \mu_V(x,x^*)(\xi)$ is $T_s$-invariant,
    
    \item [$\cdot$] $T_s$ contains two 1-dimensional subtori $\mbC_1^*, \mbC_2^*$ acting on $T^*V\times \g$ with weights $(1,-1,0), (1,0,-1)$, respectively.
    \end{itemize}

Then we show 
\begin{theorem}
Under the above assumptions, the 
$H^\BM_{G\times T_s}(\pt,\mbQ)$-module $H_{G\times T_s}^\BM(\mu_V^{-1}(0),\mbQ)$ is torsion free.
\end{theorem}
 This is actually equivalent for the restriction map to the fixed point to be an embedding. An immediate corollary of this embedding is concentration of $H_{G\times T_s}^\BM(\mu_V^{-1}(0),\mbQ)$ in even homological degrees. To prove the theorem we adapt the argument from \cite{SV20} and \cite{D22} where torsion freeness of deformed preprojective Hall algebra was proved.

Our second statement is about conditions, verified by symmetric polynomials in the image of restriction map to the fixed point. Inspired by the paper \cite{Z},  where the image of the preprojective K-theoretic Hall algebra of surfaces to a shuffle algebra was studied. It was shown that polynomials $R\in K_{GL_n(\mbC)\times (\mbC^*)^2} \simeq \mbZ[q^{\pm},q^{\prime\pm}][z_1^{\pm},\dots,z_n^{\pm}]^{S_n}$ in the image verify
\begin{align*}
    R\vert_{z_j-qz_i=0, z_k-q^\prime z_j=0}=0
\end{align*}
for any triple $\{i\neq j\neq k\}\subset \{1,\dots,n\}$. In our setting this situation corresponds to $V$ the representation space of a 1-loop quiver, acted upon by $GL_n$ by conjugation.

We generalize this result. 

Let $V$ be a finite dimensional representation of a complex reductive group $G\times T_s$ for some torus $T_s$ leaving invariant the zero-set $\mu_V^{-1}(0)$ under the $G$-equivariant moment map $\mu_V: T^*V \to \g^*$. Suppose the fixed locus $\mu_V^{-1}(0)^{T\times T_s}$ is a point, the origin. Denote the equivariant embeddings by $\pt=\mu_V^{-1}(0)^{T\times T_s}\xrightarrow{i_0} V\oplus V^* \xleftarrow{p} \mu_V^{-1}(0)$. The map $i_0$ is lci and the map $p$ is proper. Then the operations $i_0^*$ and $p_*$ in equivariant K-theory are well-defined. Denote by $W$ the Weyl group of the pair $(G,T)$. After choosing a basis in $V$ we can talk about coordinate lines in $V$ and in $V^*$. Consider two coordinate lines $l\subset V$ and $l^\prime\subset V^*$. The lines are $T\times T_s$-invariant, denote by $\chi_l,~\chi_{l^\prime}$ their $T\times T_s$-characters.

\begin{theorem}
    The image under restriction map
    \begin{align*}
        K_{G\times T_s}(\mu_V^{-1}(0))\xrightarrow{i_0^*p_*}K_{G\times T_s}(\pt)
    \end{align*}
    is contained in the ideal 
    \begin{align}
        K_{T\times T_s}(\pt)^W\cap \bigcap_{\Pi} (1 - \chi_{l}^{-1},1 - \chi_{l^\prime}^{-1})
    \end{align}
    where the intersection is taken over the set $\Pi$ of all pairs of coordinate lines $l\subset V,~l^{\prime}\subset V^*$ such that 
    \begin{align*}
        l\oplus l^{\prime}\times_{V\oplus V^*} \mu_V^{-1}(0)=l\cup_{\{0\}}l^{\prime}
    \end{align*}
    scheme-theoretically.

\end{theorem}

The divisibility conditions discussed here are called the \textit{wheel conditions}. We illustrate this theorem by writing down wheel conditions for adjoint representations of semisimple groups and irreducible representations of $SL_2$.

\section{Acknowledgements}
We thank Ben Davison for useful remarks about earlier versions of this paper and Éric Vasserot for helpful discussions. The author is partially supported by International Laboratory of Cluster Geometry NRU HSE, RF Government grant, ag. № 075-15-2021-608 dated 08.06.2021

\section{Preliminaries}\label{Prelim}

By a complex variety we mean a finite type reduced scheme over $\mbC$. When a complex variety $X$ is equipped with an action of a complex linear algebraic group $G$ we will denote by $X/G$ the quotient stack. This is an Artin stack locally of finite type over $\mbC$. We will also consider the cotangent stack $T^*(X/G)$. This 
is a 0-shifted symplectic stack, the quotient of a derived fiber product $T^*(X/G)= T^*X \times^{\g^*}\{0\}/G$ where $T^*X \xrightarrow{\mu} \g^*$ is the $G$-equivariant moment map. Its classical truncation is isomorphic to $\mu^{-1}(0)/G$.  When we discuss homology or K-theory of $T^*(X/G)$ we refer to homology or K-theory of its classical truncation.

\subsection{Derived constructible category}
For $X$ a complex variety let $D^b_c(X, \mbQ)$ be the full triangulated subcategory of derived category of constructible sheaves $D^b(Sh_c(X,\mbQ))$, whose objects are bounded complexes of sheaves on $X(\mbC)$ of $\mbQ$-vector spaces with constructible cohomology. In what follows we will write $f_*$ instead of $Rf_*$ meaning derived functors. 

The formalism of six functors in this context means the assignment 

\begin{enumerate}
    \item for every $X$ the category $D^b_c(X, \mbQ)$
    \item for every morphism $X\xrightarrow{f}Y$ two pairs of adjoint functors $(f^*,f_*), (f_!,f^!)$
    
\[\begin{tikzcd}
	{D^b_c(X, \mbQ)} & {D^b_c(Y, \mbQ)}
	\arrow["{(f_*, f_!)}", from=1-1, to=1-2]
	\arrow["{(f^*,f^!)}", shift left=3, from=1-2, to=1-1]
\end{tikzcd}\]
    \item for every $X$ and every $\mathcal{F}\in D^b_c(X,\mbQ)$ the pair of adjoint functors $(-\otimes^L\mathcal{F},RHom(\mathcal{F},-))$, endowing $D^b_c(X,\mbQ)$ with a unital symmetric monoidal structure, with unit $\mbQ_X$.
\end{enumerate}

The category $D^b_c(X,\mbQ)$ is endowed with Verdier duality functor $\mbD: D^b_c(X,\mbQ)\to D^b_c(X,\mbQ)$, $\mbD \mathcal{F} = RHom(\mathcal{F}, (X\to \pt)^!\mbQ)$. Its main property is that there is a natural isomorphism of functors $\mathrm{id}\simeq \mbD\circ\mbD$.

We recollect some important compatibilities between these functors that we use, see \cite{Achar} for details and proofs.

\begin{enumerate}
    \item  For any $X\xrightarrow{g}Y\xrightarrow{f}Z$  $(f\circ g)_*\simeq f_*\circ g_*$, $(f\circ g)^* \simeq g^*\circ f^*$ and similarly for $f_!, f^!$.
    \item $f^*$ is monoidal: for any $\mathcal{F},G$ there $f^*(\mathcal{F}\otimes G)\simeq f^*\mathcal{F}\otimes f^*G$ 
    \item There exist a natural transformation $f_!\to f_*$ that is an isomorphism for proper $f$
    \item smooth pullback: for $X\xrightarrow{f}Y$ a smooth map of relative dimension $d$ the functor $f^!$ has a simple description: there is a natural isomorphism $f^!\simeq f^*[2d]$, compatible with composition of smooth maps. Here $[-]$ is the shift functor. In particular, $j^!\simeq j^*$ for an open embedding. 

    \item\label{lcipullback} lci pulback: let $Z\xrightarrow{f}Y$ be a locally complete intersection(lci) morphism of the form 
    \[\begin{tikzcd}
	Z & X \\
	& Y
	\arrow["s"', hook, from=1-1, to=1-2]
	\arrow["p"', shift left=3, curve={height=18pt}, from=1-2, to=1-1]
	\arrow["h", from=1-2, to=2-2]
\end{tikzcd}\]
for $s$ a regular embedding of codimension $c$ that is a section of smooth map $p$(of relative dimension $c$) and $h$ a smooth map of relative dimension $d$. The condition on $h$ gives $h^!\mbQ_Y \simeq \mbQ_X[2d]$ and the condition on $s$ gives $s^!\mbQ_X = s^!p^! \mbQ_Z[-2c] = \mbQ_Z[-2c]$. Then $f^!\mbQ_Y=s^!h^!\mbQ_Y=\mbQ_Z[2(d-c)]$. Applying Verdier duality to the adjunction $f_!f^!\mbQ_Y\to \mbQ_Y$, one gets lci pullback $f^!: \mbD \mbQ_Y \to f_* \mbD\mbQ_Z[2(c-d)]$

    \item Verdier duality commutes with sheaf operations: for any $X\xrightarrow{f}Y$ there are natural isomorphisms 
    \begin{align*}
        &\mbD f_*\simeq f_! \mbD, &\mbD f_! \simeq f_* \mbD\\
        & \mbD f^*\simeq f^! \mbD, &\mbD f^! \simeq f^* \mbD
    \end{align*}

\item\label{disttrian} open-closed distinguished triangles: suppose $U\xrightarrow{i}X$ is an open embedding and $X \xleftarrow[]{j} Z$ is its closed complement. In $D^b_c(X,\mbQ)$ there are distinguished triangles
\begin{align*}
    &i_*i^! \to id \to j_*j^* \xrightarrow{+1}\\
    &j_!j^*\to id\to i_*i^*\xrightarrow{+1}
\end{align*}
\item $D^b_c(\pt,\mbQ)=D^b(\mbQ-Vect)$
    \end{enumerate}

Denote by $a_X: X\to \pt$ a map to a point. For $\mathcal{F}\in D^b_c(X,\mbQ)$ 
denote by 
$H^i(X,\mathcal{F})=H^ia_{X,*}\mathcal{F}$ its cohomology, $H^i_c(X,\mathcal{F})=H^ia_{X,!}\mathcal{F}$ compactly supported cohomology,  $H_i(X,\mathcal{F})=H^{-i}a_{X,!}\mbD\mathcal{F}$ homology and by
$H_i^{\BM}(X,\mathcal{F}) = H^{-i}(a_{X,*}\mbD\mathcal{F})$ its Borel-Moore homology. When $\mathcal{F}$ is a constant sheaf $\mbQ_X$, one recovers (singular)cohomology $H^i(X, \mbQ)$ and other invariants of $X$.

\subsection{Borel-Moore homology} 

For a complex variety $X$ one defines its Borel-Moore homology as $H_i^\BM(X,\mbQ)=H^{-i}(a_{X,*}\mbD \mbQ_X)$. It is related to the dual compactly supported cohomology as $H_i^\BM(X,\mbQ)=H^{-i}(\mbD a_{X,!} \mbQ_X)=H^i_c(X,\mbQ)^\vee$. When $X$ is proper  $H_i^\BM(X,\mbQ)=H_i(X,\mbQ)$. When $X$ is smooth $H_i^\BM(X,\mbQ) = H^{-i+2\dim X}(X,\mbQ)$. 

Let $X$ be equipped with an action of a linear group $G$, assume $X$ is quasi-projective with a fixed $G$-linearized very ample line bundle. Its equivariant compactly supported cohomology $H_{c,G}(X,\mbQ)$ is defined via the limiting construction as follows. 
We assume $G$ is a complex algebraic subgroup of $GL_n(\mbC)$ for some $n$. For $N\geq n$ denote by $\fr(n,N)$ the variety with a free $G$-action of tuples of $n$ linearly independent vectors in $\mbC^N$. The group $G$ acts freely on $V\times fr(n,N)$ by $g\cdot (v,h)=(g\cdot v,g^{-1}v)$, denote by $X_N:=X\times_G fr(n,N)$ the quotient variety.  The embedding $\mbC^N \to \mbC^{N+1}$, sending $(x_1,\dots,x_N)$ to $(x_1,\dots,x_N,0)$, induces closed embeddings $\fr(n,N)\to \fr(n,N+1)$ and $X_N\xrightarrow{i_N}X_{N+1}$.

Suppose $X$ is smooth, then varieties $X_N$ are smooth as well. Applying Verdier duality to a morphism $\mbQ_{X_{N+1}}\to i_{N,*} \mbQ_{X_N}$,
we get a morphism 
\begin{align}
i_{N,!}\mbQ_{X_N}[2 \dim X_N] \to \mbQ_{X_{N+1}}[2 \dim X_{N+1}].
\end{align}

Then, applying $(X_{N+1}\to \pt)_!$, we get a morphism in $D^b_c(\pt,\mbQ)$ 
\begin{align*}
    H_c(X_N, \mbQ)[2 \dim X_N]\to H_c(X_{N+1},\mbQ)[2 \dim X_{N+1}]. 
\end{align*}
One defines (for any $X$, possibly singular)
\begin{align*}
    H_c(X/G, \mbQ):=\colim_{N\to \infty} ~H_c(X_N,\mbQ)[2 \dim \fr(n, N)]
\end{align*}

and 
\begin{align*}
    H^\BM(X/G, \mbQ):= H_c(X/G, \mbQ)^\vee =\lim_{N\to \infty} ~H_c(X_N,\mbQ)^\vee[-2 \dim \fr(n, N)]
\end{align*}

The non-compactly supported version is defined by applying $(X_{N+1}\to \pt)_*$ to
\begin{align*}
    \mbQ_{X_{N+1}}\to i_{N,*}\mbQ_{X_N} 
\end{align*}
to get linear maps
\begin{align*}
    H(X_{N+1}, \mbQ)\to H(X_{N}, \mbQ). 
\end{align*}
One defines 
\begin{align*}
    H(X/G, \mbQ):=\lim_{N\to \infty}~H(X_{N}, \mbQ).
\end{align*}

Assume $X/G\simeq Y/H$ is an isomorphism of stacks. By [Prop 16,\cite{EG}], there is an isomorphism 
$H^\BM_{i+2\dim G, G}(X,\mbQ)\simeq H^\BM_{i+2\dim H, H}(Y,\mbQ)$. Then one can relate homology of a stack with equivariant homology by 
\begin{align}\label{BMeq}
H^\BM_i(X/G,\mbQ)=H^\BM_{i+2\dim G, G}(X,\mbQ),  
\end{align}
\begin{align*}
H^i(X/G,\mbQ)=H^i_G(X,\mbQ).
\end{align*}

We collect some standard properties of the functors $H^\BM$ and $H^\BM_G$.

\textbf{Properties}

\begin{itemize}

    \item  proper pushforward: let $X\xrightarrow{f}Y$ be a  proper map. Then the Verdier dual of the adjunction $\mbQ_{Y}\to f_{*}\mbQ_{X}$ gives $H_{i}^\BM(X,\mbQ)\to H_{i}^\BM(Y,\mbQ)$
    \item  lci pullback: let $X\xrightarrow{f}Y$ be lci of the form $(\ref{lcipullback})$ of relative dimension $d-c$. Then the map $f^! :\mbD\mbQ_Y\to f_*\mbD\mbQ_X[2(c-d)]$ gives $H_i^\BM(Y,\mbQ)\to H^\BM_{i-2(c-d)}(X,\mbQ)$

    \item refined pullback: suppose the square is Cartesian
    \[\begin{tikzcd}
	{X^{\prime}} & {Y^{\prime}} \\
	X & Y
	\arrow["h", from=1-1, to=1-2]
	\arrow["l", from=1-1, to=2-1]
	\arrow["p", from=1-2, to=2-2]
	\arrow["f", from=2-1, to=2-2]
\end{tikzcd}\] with $f$ lci of relative dimension $d$. Composing $p^!$ with lci pullback $f^!: \mbD\mbQ_Y\to f_*\mbD\mbQ_X[2d]$, one gets $\mbD\mbQ_{Y^\prime}\to h_*\mbD\mbQ_{X^\prime}[2d]$. Taking cohomology, one gets 
\begin{align*}
    (f,h)^!:H_i^\BM(Y^\prime,\mbQ)\to H^\BM_{i+2d}(X^\prime,\mbQ).
\end{align*}
    If $h$ is also lci, then $(f,h)^!=h^!$.
\end{itemize}

    The same functorialities hold for the functor $H^\BM_G$.

   \begin{itemize} 
    \item  open-closed long exact sequence:
    the first distinguished triangle in $(\ref{disttrian})$ applied to $\mbD\mbQ_X$ gives a long exact sequences
    \begin{align*}
        \to H_i^\BM(Z,\mbQ)\to H_i^\BM(X,\mbQ)\to H_i^\BM(U,\mbQ)\to H_{i-1}^\BM(Z,\mbQ)\to
    \end{align*}

Assuming $Z\xrightarrow{i} X$ is a $G$-equivariant closed embedding, one gets a long exact sequence of equivariant BM-homology
    
    \item homotopy invariance: $H_i^\BM(\mbR^n,\mbQ)=\mbQ$ for $i=n$ and $0$ otherwise

\item Intersection pairing: 
For any $\mathcal{F},G\in D_c^b(X,\mbQ)$ and any $f:X\to Y$ one has a morphism in $D^b_c(Y,\mbQ)$
\begin{align*}
f_*\mathcal{F}\otimes f_*G \to f_*f^*(f_*\mathcal{F}\otimes f_*G) = f_*(f^*f_*\mathcal{F}\otimes f^*f_*G)\to f_*(\mathcal{F}\otimes G),
\end{align*}
since $f^*$ is monoidal.

For any $X$, applying to $\mathcal{F}=G=\mbQ_X$, one gets a bilinear pairing
\begin{align*}
H^i(X,\mbQ)\otimes H^j(X,\mbQ) \to H^{i+j}(X,\mbQ). 
\end{align*}

For $X$ smooth, applying to $\mathcal{F}=G=\mbD\mbQ_X$ one gets a bilinear pairing 
\begin{align*}
H_i^\BM(X,\mbQ)\otimes H_j^\BM(X,\mbQ) \to H_{i+j-2\dim X}^\BM(X,\mbQ) 
\end{align*}

The bilinear pairing in equivariant case($X$ is again smooth) is defined similarly
\begin{align*}
H_i^\BM(X/G,\mbQ)\otimes H_j^\BM(X/G,\mbQ) \to H_{i+j-2\dim X/G}^\BM(X/G,\mbQ)
\end{align*}
as the composition
\begin{align*}
    H_i^\BM(X/G,\mbQ)\otimes H_j^\BM(X/G,\mbQ) = \lim_{N,N^\prime} H^\BM_{i+2d_N}(X_N)\otimes H^\BM_{j+2d_{N^\prime}}(X_{N^\prime})\to\\ \to\lim_M H^\BM_{i+2d_M}(X_M)\otimes H^\BM_{j+2d_M}(X_M)\to \lim_M H^\BM_{i+j+2d_M-2\dim X/G}(X_M)= H^\BM_{i+j-2\dim X/G}(X/G,\mbQ),
\end{align*}
where $d_\star =\dim \fr(n,\star)$.

\item When $X$ is proper, the natural map $H_i(X/G,\mbQ)\to H_i^\BM(X/G,\mbQ)$ is an isomorphism. When $X$ is smooth, the natural map 
\begin{align}\label{PD}
H^{-i+2\dim X/G}(X/G,\mbQ)\to H_i^\BM(X/G,\mbQ), ~ \alpha \to \alpha \cap [X]
\end{align}
is an isomorphism, called the Poincaré duality. In particular, 
$H^i_G(\pt,\mbQ)\simeq H_{-i, G}^\BM(\pt,\mbQ)$.

\end{itemize}

\subsection{K-theory}

Let $X$ be a complex quasi-projective variety with an action of a complex linear group $G$.
One defines the equivariant K-theory $K_G(X)$ of $X$ as the Grothendieck group of an abelian category of $G$-equivariant coherent sheaves on $X$. We list below some of its properties that will be used and refer to \cite{CG} for details.

\textbf{Properties}\label{prsKth}
\begin{itemize}
    \item K-theory of a point: A coherent sheaf on a pont is a finite dimensional complex $G$-representation. Denote by $R_G$ the ring of characters of $G$, then $K_G(\pt)=R_G$. For any $X$, $K_G(X)$ is a module over $K_G(\pt)$.

    \item pullback: let $X$ and $Y$ be smooth quasi-projective varieties and $Y\xrightarrow{i} X$ be a closed $G$-equivariant embedding. One defines $f^*: K_G(X)\to K_G(Y)$ as a finite sum 
    \begin{align*}
        f^*[\mathcal{F}]=\sum_{i\geq 0}(-1)^i[L^if^* \mathcal{F}] 
    \end{align*}
where to compute the $G$-equivariant sheaves $L^if^* \mathcal{F}=Tor_i^{\mathcal{O}_X}(f^*F,\mathcal{O}_Y)$ one picks a (finite)locally free $G$-equivariant resolution $F^\bullet$ of (non-derived)$f^*\mathcal{F}$
\begin{align*}
    \dots \to F^1\to F^0\to f^*\mathcal{F}\to0
\end{align*}
and computes the cohomology
$L^if^* \mathcal{F} = H^i(F^\bullet\otimes_{\mathcal{O}_X} \mathcal{O}_Y)$.
\item refined pullback: given a Cartesian diagram 
\[\begin{tikzcd}
	{X^\prime} & X \\
	{Y^\prime} & Y
	\arrow["{g^\prime}", from=1-1, to=1-2]
	\arrow["{f^\prime}", from=1-1, to=2-1]
	\arrow["f", from=1-2, to=2-2]
	\arrow["g", from=2-1, to=2-2]
\end{tikzcd}\]
with $f$ lci between smooth varieties, the refined pullback $f^!: K_G(Y^\prime)\to K_G(X^\prime)$ is defined as a finite sum
\begin{align*}
    f^!([\mathcal{F}]) = \sum_{i\geq 0}(-1)^i[Tor_i^{\mathcal{O}_Y}(\mathcal{O}_X,\mathcal{F})].
\end{align*}

\item base change \cite[Lemma 2.5]{Z}: given a Cartesian diagram 
\[\begin{tikzcd}
	{X^\prime} & X \\
	{Y^\prime} & Y
	\arrow["{g^\prime}", from=1-1, to=1-2]
	\arrow["{f^\prime}", from=1-1, to=2-1]
	\arrow["f", from=1-2, to=2-2]
	\arrow["g", from=2-1, to=2-2]
\end{tikzcd}\]
with $f$ lci and $g$ proper, then $f^*g_*=g^\prime_*f^! :K_G(Y^\prime)\to K_G(X)$.

\item proper pushforward: given a proper $G$-equivariant map $f: X\to Y$ between quasi-projective varieties
    and a class $[\mathcal{F}]\in K_G(X)$, one defines $f_*: K_G(X)\to K_G(Y)$ by
    \begin{align*}
        f_*[\mathcal{F}]=\sum_{i\geq 0}(-1)^i [R^if_*\mathcal{F}]
    \end{align*}
where the sum is finite;

     \item an isomorphism  $K(X/G)\simeq K_G(X)$.
\end{itemize}
We will often use the shorthands $H_G=H(BG,\mbQ)$ and $K_G=K(BG)$.

\section{Vanishing cycles}\label{vancyc}

In this section we recall the definition of vanishing cycles and the statement of dimensional reduction from \cite{D16} to use it in Section \ref{HallInd}.

Let $X$ be a complex algebraic variety and $f: Y\to \mbC$ be a regular function on it. Denote by $X^*:=f^{-1}(\mbC^*)$ and $X_0:=f^{-1}(0)$. Consider the diagram with Cartesian squares
\[\begin{tikzcd}
	{\widetilde{X^*}} & {X^*} & X & {X_0} \\
	{\mbC=\widetilde{\mbC^*}} & {\mbC^*} & \mbC & 0
	\arrow["\pi", from=1-1, to=1-2]
	\arrow[from=1-1, to=2-1]
	\arrow["j", hook, from=1-2, to=1-3]
	\arrow[from=1-2, to=2-2]
	\arrow["f", from=1-3, to=2-3]
	\arrow["{{i_0}}"', hook', from=1-4, to=1-3]
	\arrow[from=1-4, to=2-4]
	\arrow["{\mathrm{exp}}", from=2-1, to=2-2]
	\arrow[hook, from=2-2, to=2-3]
	\arrow[hook', from=2-4, to=2-3]
\end{tikzcd}\]

One defines \cite{Dimca} the nearby cycle functor $D_c^b(X,\mbQ)\to D_c^b(X_0,\mbQ)$ by
\begin{align*}
    \psi_f:=i_0^*(\pi\circ j)_*(\pi\circ j)^*
\end{align*}

The vanishing cycle functor is defined as the cone of a canonical morphism 
\begin{align*}
    i_0^*\xrightarrow{can} \psi_f\to\phi_f\to i_0^*[1]
\end{align*}

We give several examples when  these functors are easy to compute

\begin{example}
     For $f=0$ one has $\psi_f=0$ and $\phi_f=[1]$
\end{example}

\begin{example}
    $\mathcal{F} = i_{0,*}V$ the skyscaper sheaf at $0\xrightarrow{i_0} \mbC$ and $f=t: \mbC\to \mbC$ is the coordinate. Then again $\psi_t \mathcal{F}=0$, $\phi_t\mathcal{F} = V[1]$. 
\end{example}

\begin{example}
    Consider $Y=\mbC$ together with $\mbC\xrightarrow{f}\mbC,~z\mapsto z^n$. We compute the functors $\psi_f, \phi_f$ applied to a constant sheaf $\mbQ$. By the standard property, stalk cohomology of the nearby cycles compute cohomology of the Milnor fiber
    \begin{align*}
H^i(MF_p,\mbQ)=\mathcal{H}^i(\mbC,\psi_f\mbQ)_p=H^i((\psi_f\mbQ)_p).
    \end{align*}
    In our situation, the fiber at zero is one point, and for $\delta\neq 0~$ $f^{-1}(\delta)$ consists of $n$ points. Thus the Milnor fiber at 0 consists of $n$ points.  
    We have $\psi_f \mbQ = \mbQ_0^n$ is a sheaf supported at 0 with fiber the vector space of dimension $n$. The canonical map $\mbQ_0=i_0^*\mbQ\to \mbQ_0^n$ is diagonal embedding $1\mapsto (1,\dots,1)$. Then $\phi_f\mbQ = Cone[\mbQ_0 \xrightarrow{(1,\dots,1)}\mbQ^n_0]\simeq \mbQ_0^{n-1}$ with cohomology $H^0(\phi_f\mbQ) = \mbQ^{n-1}$, $H^{-1}(\phi_f\mbQ)=0$.
\end{example}

\begin{example}
    Consider $\mathcal{F} = j_*\mathcal{L}$ where $\mathcal{L}$ is a local system on $\mbC^*\xrightarrow{j}\mbC \xrightarrow{t} \mbC$ given by its stalk $V$  and its endomorphism $T\in \End(V)$, coming from monodromy representation $\mbZ\simeq\pi_1(\mbC^*)\to \mathrm{GL(V)},~1\mapsto T$. 
    We have $\psi_t\mathcal{F} = V$ and $i_0^*\mathcal{F} \simeq \ker (T-id)$ and $\phi_t\mathcal{F} \simeq V/\ker (T-id)\simeq \im(T-id)$.
\end{example}
 
Recall an equivalent definition of vanishing cycles from \cite{D16}, used to formulate dimensional reduction. 

Denote by $X_+:=f^{-1}(\mbR_{>0})$ and $X_0 := f^{-1}(0)$. One defines the nearby cycle functor by
\begin{align*}
    \psi_f:=(X_0\to X)_*(X_0\to X)^*(X_+\to X)_*(X_+\to X)^*
\end{align*}
from the category $D_c^b(X,\mbQ)$ to itself, of bounded complexes of sheaves of $\mbQ$-vector spaces with constructible cohomologies.

The vanishing cycle functor is defined as the cone of a natural morphism
\begin{align*}
    \phi_f:= Cone((X_0\to X)^*(X_0\to X)_*\to \psi_f)
\end{align*}
that is for any $\mathcal{F}\in D_c^b(X,\mbQ)$ there exists an exact triangle
\begin{align*}
    \phi_f \mathcal{F}[-1]\to (X_0\to X)_*(X_0\to X)^*\mathcal{F} \to \psi_f\mathcal{F} \to \phi_f\mathcal{F}.
\end{align*}


\textbf{Properties}\label{propertiesofphi}

The nearby and vanishing cycles verify remarkable properties, making them manageable to work with. We recall some of them, that will be used. Consider  $X^{\prime}\xrightarrow{j} X\xrightarrow{f}\mbC$ the morphism of complex varieties $j$ followed by a function $f$.

\begin{itemize}
    \item Commutes with Verdier duality
\begin{align*}
    \phi_f\mbD \simeq \mbD \phi_f
\end{align*}
    \item Supported at the intersection of the critical locus  and the zero-fiber $X_0$: 
    \begin{align*}
        \Supp ~H^k(\phi_f \mbQ) \subseteq \Crit(f)\cap X_0.
    \end{align*}
    In our applications it will actually be supported at the critical locus $\mathrm{Crit(f)}\subset X_0$.

    \item Commutes with  pushforwards of closed embeddings:  for $j$ a closed embedding the natural transformation
    \begin{align*}
        \phi_fj_* \to j_*\phi_{fj}j^*j_*\simeq j_*\phi_{fj}
    \end{align*}
    is an equivalence. It is obtained by composing the natural transformation
    \begin{align*}
        \phi_f\to j_*\phi_{fj}j^*,
    \end{align*}
    with $j_*$.
      
    \item Commutes with smooth pullbacks: for $j$ smooth the natural transformation $j^*\phi_f \to j^*j_*\phi_{fj}j^*\simeq \phi_{fj}j^*$ is an equivalence.

    \item For $j$  an affine fibration a natural transformation
    \begin{align}\label{afffib}
        \phi_fj_!j^* \to j_!\phi_{fj}j^*
    \end{align}
    is an equivalence by [Corollary 2.4, \cite{D16}]
\end{itemize}
In general $\phi_f$ does not commute with pushforwards along open embeddings, as shows the following simple example. Consider $\mbC^*\xrightarrow{j}\mbC\xrightarrow{id}\mbC$ and consider a non-trivial local system $\mathcal{L}$ on $\mbC^*$. Then $0\neq(\phi_{f}j_*\mathcal{L})_0 \to  (j_*\phi_{fj}\mathcal{L})_0=0$ since $\Supp (j_*\phi_{fj}\mathcal{L}) \subseteq \Supp (\phi_{fj}\mathcal{L})$ is contained in the singular locus of the fiber at $0$, which is empty.

In what follows we often write $\phi_f$ instead of $\phi_f\mbQ[-1]$.

\subsection{Equivariant vanishing cycles}

We use the definition of the vanishing cycle cohomology of a stack from \cite{D16} to define $H_{c}(V/G,\phi_f)$ and $H(V/G,\phi_f)$ for $f:V \to \mbC$ a $G$-invariant function on representation $V$.

We assume $G$ is a complex algebraic subgroup of $\mathrm{GL}_n(\mbC)$ for some $n$. For $N\geq n$ denote by $\fr(n,N)$ the variety with a free $G$-action of tuples of $n$ linearly independent vectors in $\mbC^N$. The group $G$ acts on $V\times fr(n,N)$ by $g\cdot (v,h)=(g\cdot v,g^{-1}v)$, denote by $V_N:=V\times_G fr(n,N)$ the quotient variety. Denote by $f_N$ the induced function on $V_N$. One associates to it a vanishing cycle complex $\phi_{f_N}\mbQ_{V_N}$. The embedding $\mbC^N \to \mbC^{N+1}$, sending $(x_1,\dots,x_N)$ to $(x_1,\dots,x_N,0)$, induces closed embeddings $\fr(n,N)\to \fr(n,N+1)$ and $V_N\xrightarrow{i_N}V_{N+1}$ with $f_{N+1}i_N = f_N$.  
We have a morphism in $D_c^b(V_{N+1},\mbQ)$
\begin{align}\label{m1}
i_{N,!}\phi_{f_N}\mbQ_{V_N}[2 \dim V_N] \to \phi_{f_{N+1}}\mbQ_{V_{N+1}}[2 \dim V_{N+1}]
\end{align}
coming as a composition 
\begin{align*}
  &i_{N,!}\phi_{f_N}\mbQ_{V_N}[2 \dim V_N] \simeq i_{N,!}\phi_{f_N} \mbD\mbQ_{V_N}\simeq \\&\phi_{f_{N+1}}i_{N,!}\mbD\mbQ_{V_N}\to  \phi_{f_{N+1}} \mbD \mbQ_{V_{N+1}}\simeq \phi_{f_{N+1}}\mbQ_{V_{N+1}}[2\dim V_{N+1}]  
\end{align*}
where we used that for a smooth variety $V_N$ there is an isomorphism $\mbQ_{V_N}[2 \dim V_N] \to \mbD \mbQ_{V_N}$; $i_{N,*} \simeq i_{N,!}$ since $i_N$ is proper; an isomorphism $\phi_{f_{N+1}}i_{N,*} \simeq i_{N,*} \phi_{f_{N}}$ for a closed embedding $i_N$. 

Then, applying $(V_{N+1}\to \pt)_!$ to $(\ref{m1})$, we get a morphism in $D^b_c(\pt,\mbQ)$ 
\begin{align*}
    H_c(V_N,\phi_{f_N})[2 \dim V_N]\to H_c(V_{N+1},\phi_{f_{N+1}})[2 \dim V_{N+1}]. 
\end{align*}
One defines 
\begin{align*}
    H_c(V/G, \phi_f):=\colim_{N\to \infty} ~H_c(V_N,\phi_{f_N})[2 \dim \fr(n, N)].
\end{align*}
The definition is well defined  since $H_c(V_N,\phi_{f_N})[2 \dim \fr(n, N)]$ stabilizes in each cohomological degree.

The non-compactly supported version is defined by considering morphisms 
\begin{align*}
    \phi_{f_{N+1}}\mbQ_{V_{N+1}}\to \phi_{f_{N+1}}i_{N,*}\mbQ_{V_N} \simeq i_{N,*}\phi_{f_N}\mbQ_{V_N}
\end{align*}
and applying $(V_{N+1}\to \pt)_*$ to get linear maps
\begin{align*}
    H(V_{N+1}, \phi_{f_{N+1}})\to H(V_{N}, \phi_{f_{N}}). 
\end{align*}
One defines 
\begin{align*}
    H(V/G, \phi_f):=\lim_{N\to \infty}~H(V_{N}, \phi_{f_{N}}).
\end{align*}

In Section (\ref{TF}) we will use the $H_G(\pt,\mbQ)$-module structure on the vector space $H_{c,G}(X,\phi_f)^\vee$, constructed in [Section 2.6, \cite{D16}] for any complex variety $X$ with a $G$-invariant function $X\xrightarrow{f}\mbC$.

\subsection{Dimensional reduction}

Let $X$ be complex variety and denote by $\pi: X\times \mbC^n \to X$ the projection. Let $f: X\times \mbC^n \to \mbC$ be a $\mbC^*$-equivariant regular function, with weight $(0,1)$ on the source and weight 1 on the target. Then $f$ has the form $f=\sum_{i=1}^n f_ix_i$ where $x_1,\dots,x_n$ are coordinates on $\mbC^n$ and $f_i$ are functions on $X$. Denote the set of common zeroes by $Z=V(f_1,\dots,f_n)$. Then $Z$ is the subset of such points $x\in X$ that  $\pi^{-1}(x)\subset f^{-1}(0)$. Denote by $i:Z\to X$ the closed embedding.
The dimensional reduction states an isomorphism of functors \cite[Theorem A.1]{D16}
\begin{align*}
\pi_!\phi_f \pi^*[-1] \simeq \pi_!\pi^*i_*i^*
\end{align*}

Applying to the constant sheaf on $X$ and taking the derived global sections on both sides, one gets an isomorphism in compactly supported cohomology
\begin{align*}
H_c^{i-1}(X\times \mbC^n, \phi_f \mbQ) \simeq H_c^{i-2n}(Z,\mbQ)
\end{align*}

\subsection{Equivariant dimensional reduction}\label{eqdimred}

The dimensional reduction isomorphism extends to quotient stacks.  Assume also that $X$ is a complex algebraic $G$-variety for $G$ a complex algebraic group and 
$\pi: X\times \mbC^n\to X$ is a $G$-equivariant vector bundle over $X$. Assume
$f:X\times \mbC^n\to \mbC$ is a $G$-invariant function, again $\mbC^*$-equivariant with weight $(0,1)$ on the source and weight 1 on the target.  The dimensional reduction asserts an isomorphism in compactly supported cohomology \cite[Corollary A.9]{D16}
\begin{align*}
H_c^{i-1}(X\times\mbC^n/G, \phi_f \mbQ) \simeq H_c^{i-2n}(Z/G,\mbQ).
\end{align*}

Below we specialize the equivariant dimensional reduction isomorphism to situations that we will need.

\begin{example}\label{ex1dimred}
    Let $V$ be a $G$-representation, then $G$ acts on $X=T^*V\times \g$. Let $X$ be a vector space $T^*V\times\g$ equipped with an action of $G\times T_s$ for some auxiliary torus $T_s$(for example trivial) such that the $\mbC^*$-equivariant function $f:X \to \mbC$ given by $f(x,x^*,a) = \mu_V(x,x^*)(a) = \langle x^*, a.x \rangle$ is $G\times T_s$-invariant. Here we consider the $\mbC^*$-equivariance of $f$ with weights $(0,0,1)$. Consider a trivial $G\times T_s$-equivariant fibration $T^*V\times \g\to T^*V$. In this situation $Z=\mu_V^{-1}(0)$. Then we have
    \begin{align}\label{dimred1}
        H_c^{i-1}(T^*V\times\g/G\times T_s, \phi_f \mbQ) \simeq H_c^{i-2\dim \g}(\mu_V^{-1}(0)/G\times T_s,\mbQ).
    \end{align}
\end{example}

As a particular case of this example we consider the case of a quiver with potential \cite[Section A.3]{D16} . 

    Let $Q=(Q_0,Q_1)$ be a quiver and $\mbC Q$ be its path algebra. Denote by $Q^{op}$ the opposite quiver, with the same set of vertices as $Q$ but with all arrows reversed; adding the opposite arrow $a^*$ to each arrow $a\in Q_1$ of $Q$ one gets  $\bar{Q}$ the doubled quiver; adding a loop $\omega_i$ at each vertex $i$ of the doubled quiver one gets $\tilde{Q}$ the tripled quiver. 
    
    Fix $\mathbf{d}\in \mbN ^{\vert Q_0\vert}$ the dimension vector. We take the group $G$ to be $\mathrm{G}_{\mathbf{d}} := \prod_i \mathrm{GL}_{d_i}(\mbC)$, its Lie algebra $\gl_d=\oplus_i \gl_{d_i}$ is identified with its dual by the trace. Take the representation $V$ to be the representation space of the quiver
    \begin{align*}
        V=\bigoplus_{i\to j }\Hom(\mbC^{d_i}, \mbC^{d_j}).
    \end{align*}
    Here element $(g_i\in \mathrm{GL}_{d_i}(\mbC))_{i\in Q_0}$ acts on $(M_{ij}:\mbC^{d_i}\to \mbC^{d_j})_{i\to j \in Q_1}$ by conjugation 
    \begin{align*}
        M_{ij}\mapsto g_jM_{ij}g_i^{-1}.
    \end{align*}
    The dual representation $V^*$ is naturally identified with representation space of the opposite quiver, where all arrows are reversed
    \begin{align*}
       V^*= \bigoplus_{j\to i}\Hom(\mbC^{d_j}, \mbC^{d_i}).
    \end{align*}

  The moment map 
   \begin{align*}
     \mu_{\dd}: T^*V \to \gl_{\dd}^*\simeq \gl_{\dd} 
    \end{align*}
sends ($T^*V\simeq V\times V^*$)
\begin{align*}
    (x_a,x_{a^*})_{a:i\to j\in Q_1} \mapsto 
    \sum_{a:i\to j\in Q_1} [x_a,x_{a^*}]:= \sum_{a:i\to j\in Q_1} (x_ax_{a^*},-x_{a^*}x_a)
\end{align*}
where $x_ax_{a^*}\in \gl_{d_j}$ and $x_{a^*}x_a\in \gl_{d_i}$. 

The stack $\mu_{\dd}^{-1}(0)/\mathrm{G}_{\dd}$ is is isomorphic to the stack $\mathfrak{M}_{\Pi_Q,\dd}$ of $\dd$-dimensional represenations of a \textit{preprojective algebra} of $Q$
\begin{align}\label{prep}
    \Pi_Q := \mbC \bar{Q}/\sum_{a\in Q_1} [a,{a^*}]
\end{align}

 Let $W \in \mbC Q/[\mbC Q, \mbC Q]$ be the formal linear combination of cycles in the  quiver(here one takes the quotient of a vector space by the vector subspace spanned by commutators of elements in $\mbC Q$). Given $W=\sum_ia_i C_i$, the linear combination of cycles $C_i$ in $Q$, and a dimension vector $\mathbf{d}$ one defines a function 
 \[\begin{tikzcd}
	{\Tr_{\mathbf{d}}(W): V \to \mbC} & {\Tr_{\mathbf{d}}(W)(\rho) = \sum_i a_i \Tr(\rho(C_i)).} & {}
\end{tikzcd}\] This function is $\mathrm{G}_{\mathbf{d}}$-invariant because of invariance of the trace under conjugation, and so descends to a function on a quotient.

In particular, given a tripled quiver $\tilde{Q}$ and the canonical qubic potential $$\tilde{W}=\sum_{i\in Q_0}\omega_{i} \sum_{a\in Q_1}[a,a^*] \in \mbC \tilde{Q}/[\mbC \tilde{Q}, \mbC \tilde{Q}],$$ one defines a function $\Tr_{\mathbf{d}}(\tilde{W}): T^*V\times \g/G \to \mbC$. In this situation $$Z_{\mathbf{d}}=\{(x_a,x_a^*)_{a\in Q_1}\in T^*V: \sum_{a\in Q_1}[x_a,x_a^*]=0\}$$ is the commuting variety. One the other hand, the equivariant vanishing cycle complex $\phi_{\Tr_{\mathbf{d}}(\tilde{W})}\mbQ$ is supported on the critical locus stack
\begin{align*}
\mathrm{Crit}(\Tr_{\mathbf{d}}(\tilde{W}))/\mathrm{G}_{\dd}\simeq \mathfrak{Jac}_{(\tilde{Q},\tilde{W}),\dd}
\end{align*}
which is isomorphic to the stack $\mathfrak{Jac}_{(\tilde{Q},\tilde{W}),\dd}$ of $\dd$-dimensional representations of \textit{Jacobi algebra}
\begin{align*}
J_{(\tilde{Q},\tilde{W})} = \mbC \tilde{Q}/(\partial_a \tilde{W}: a\in \tilde{Q}_1).
\end{align*}

The pair $(\tilde{Q},\tilde{W})$ admits a \textit{cut} given by grading $\nu: \tilde{Q}\to \mbZ_{\geq0}$ 
\[\begin{tikzcd}
	{\nu(a)=0} & {\nu(a^*)=0} & {\nu(\omega_i)=1.}
\end{tikzcd}\]

Then for the pair $(\tilde{Q},\tilde{W})$ the dimensional reduction isomorphism (\ref{dimred1}) specializes to 
\begin{align*}
    H^{i-1}_c(\mathfrak{Jac}_{(\tilde{Q},\tilde{W}),\dd}, \phi_{\Tr_{\mathbf{d}}(\tilde{W})}\mbQ) \simeq H_c^{i-2\dd^2}(\mathfrak{M}_{\Pi_Q,\dd},\mbQ).
\end{align*}
In particular, for $Q$ a one-loop quiver the prepojective algebra is $\mbC[x,y]$, and the Jacobi algebra for the potential $\tilde{W} = a[b,c]$ is $\mbC[x,y,z]$. The stacks of $d$-dimensional representations of these algebras are the stacks of length $d$ coherent sheaves on $\mbA^2$ and $\mbA^3$, respectively
\begin{align*}
&\mathfrak{M}_{\mbC[x,y],d}\simeq\mathrm{Coh}_d(\mbA^2) 
& \mathfrak{Jac}_{(Q,W),d}\simeq \mathrm{Coh}_d(\mbA^3).
\end{align*}

In this situation the theorem states an isomorphism
\begin{align*}
    H^{i-1}_c(\mathrm{Coh}_d(\mbA^3), \phi_{\Tr_{d}(a[b,c])}\mbQ) \simeq H_c^{i-2d^2}(\mathrm{Coh}_d(\mbA^2),\mbQ).
\end{align*}

\begin{example}\label{ex2dimred}
    Let $X$ and $f$ be as in Example \ref{ex1dimred}.
    Consider now a a trivial $G$-equivariant fibration to another base $T^*V\times \g\to V\times \g$. Assume now the $\mbC^*$ acts on $V\times V^* \times \g$ with weights $(0,1,0)$. We consider the same  function $f$ but with new $\mbC^*$-equivariance. Now $Z$ is the set $\{(x,a)\in V\times \g : a.x = 0\}$. 
    Then we have
    \begin{align}\label{dimred2}
        H_c^{i-1}(T^*V\times\g/G\times T_s, \phi_f \mbQ) \simeq H_c^{i-2 \dim V}(\{(x,a)\in V\times \g : a.x = 0\}/G\times T_s,\mbQ).
    \end{align}
\end{example}

\section{Hall induction}\label{HallInd}

\subsection{Dynamical method}
Let $G$ be a reductive group and $V$ be its representation. Let $T\subseteq G$ be a maximal torus. Denote by $X_*(T)=\Hom_\mbZ(\mbG_m,T)$ the lattice of cocharacters.

First, we recall the dynamical method of assigning a parabolic and Levi subgroups of $G$ to a cocharacter $\lambda: \mbG_m \to T$. These subgroups come together with naturally associated representations, see \cite{Milne}.

Namely, let $G$ be a reductive group and $T\subseteq G$ be a maximal torus. Let $\lambda: \mbG_m \to T$ be a cocharacter. To it one associates a parabolic $P_\lambda$, Levi $L_\lambda$ and unipotent $U_\lambda$ subgroups
\begin{align*}
   & P_{\lambda} = \{g\in G : \lim_{t\to 0} \lambda(t)g\lambda(t)^{-1} \text{ exists }\},\\
    & L_{\lambda} = P^{\lambda} \cap P^{\lambda^{-1}} = \{g\in G : \lambda(t)g\lambda(t)^{-1} = g ~\forall t\}, \\
    & U_{\lambda} = \{g\in G: \lim_{t\to 0} \lambda(t)g\lambda(t)^{-1} \text{ exists and equals to 1}\}
\end{align*}

The parabolic $P_\lambda$ naturally acts on $V^{\lambda\geq 0}$ and Levi $L_\lambda$ acts on $V^\lambda$, where  
\begin{align*}
&V^{\lambda\geq 0}=\{v\in V : \lim_{t\to 0} \lambda.v \text{ exists}\},\\
&V^{\lambda} = \{v\in V: \lambda(t).v = v ~\forall t\}.
\end{align*}

Denote by $\p_{\lambda},~ \el_{\lambda},~ \U_{\lambda}$ their Lie algebras. 
The condition for the limit $\lim_{t\to0}\lambda(t)g\lambda(t)^{-1}$ to exists means the following: the morphism 
\begin{align*}
\mbG_m \to G, ~t\mapsto \lambda(t)g\lambda(t)^{-1} 
\end{align*}
should extend to a morphism from $\mbA^1$.
Similarly, for the limit $\lim_{t\to 0}\lambda.v$ to exists means the morphism
\begin{align*}
    \mbG_m\to V, ~t\mapsto \lambda(t).v
\end{align*}
should extend to a morphism from $\mbA^1$.

It is clear from the definition that the maximal torus $T$ is subgroup of both $P_{\lambda}$ and $L_{\lambda}$.

\begin{theorem}\cite[Chapter 21, Section (i)]{Milne}
    Let $\lambda$ be a cocharacter of $G$. Then $P_{\lambda}$ is a parabolic subgroup of $G$, and every parabolic subgroup of $G$ is of this form
\end{theorem}

To define a parabolic induction, the following order on cocharacters was considered in \cite{H25} 
\begin{align}\label{ord}
    \lambda \preceq \nu \iff \begin{cases}
        V^{\lambda} \subseteq V^{\nu},\\
        \el_{\lambda} \subseteq \el_{\nu}\\
    \end{cases}
\end{align}

It defines an equivalence relation on a set $X_*(T)$
\begin{align*}
    \lambda \sim \nu \iff \lambda\preceq \nu \text{ and }\lambda \succeq \nu
\end{align*}

 \begin{example}
     Consider $V=\gl_4$ the adjoint representation of $GL_4$ and $T\subset GL_4$ the standard maximal torus. Let $\lambda =(1,2,2,3), \nu = (2,2,2,1) \in X_*(T)$ two cocharacters.  
     We have 
     \begin{align*}
     P_{\lambda}=\begin{pmatrix}
    * &  &&\\
    * & * &*&\\
    *&*&*&\\
    *&*&*&*
\end{pmatrix},&& L_{\lambda}=\begin{pmatrix}
    * &  &&\\
     & * &*&\\
    &*&*&\\
    &&&*
\end{pmatrix},&&
U_{\lambda}=\begin{pmatrix}
    1 &  &&\\
    * & 1 &&\\
    *&&1&\\
    *&*&*&1
\end{pmatrix}\\
        P_{\nu}=\begin{pmatrix}
    * & * &*&*\\
    * & * &*&*\\
    *&*&*&*\\
    &&&*
\end{pmatrix},&&
        L_{\nu}=\begin{pmatrix}
    * & * &*&\\
    * & * &*&\\
    *&*&*&\\
    &&&*
\end{pmatrix},&& U_{\nu}=\begin{pmatrix}
    1 &  &&*\\
     & 1 &&*\\
    &&1&*\\
    &&&1
\end{pmatrix}\\
\end{align*}
     acting on their adjoint representations
     $V^{\lambda \geq 0} = \p_{\lambda}$, etc.  We see $V^{\lambda}=\el_{\lambda}\subseteq V^{\nu}=\el_{\nu}$, so $\lambda \preceq \nu$.
 \end{example}

One considers the moment maps $\mu_V, \mu_{\geq \lambda}, \mu_{\lambda}$ with their zero-levels $\mu_V^{-1}(0), \mu_{\lambda\geq 0}^{-1}(0), \mu_{\lambda}^{-1}(0)$ invariant under the action of $G, P_{\lambda}, L_{\lambda}$, respectively.

\subsection{KS critical Hall induction}
\label{KScrit}

For any representation $V$ of $G$, a function $f$ on $V/G$ and $\lambda\in X_*(T)$, denote by 
\begin{align}\label{comp}
\mathcal{H}_{V,f,\lambda}:=H_{c,L_\lambda}(V^\lambda, \phi_{f_\lambda})^\vee[-\dim~V^\lambda/L_\lambda]
\end{align}
the shifted dual of the compactly supported equivariant cohomology.

\textbf{Assumption}
We will assume that for any $\lambda\in X_*(T)$ the following holds
\begin{align}\label{ass}
\dim ~V^\lambda/L_\lambda - \dim ~V/G = 2(\dim~V^{\lambda\geq 0}/P_\lambda - \dim~V/G). 
\end{align}
This assumption will be used for induction to preserve cohomological degree.

For $\lambda, \nu\in X_*(T)$ such that $\lambda \preceq \nu$ we define in steps the induction map
\begin{align}\label{KSHall}
    \mathcal{H}_{V,f,\lambda}\to \mathcal{H}_{V,f,\nu}.
\end{align}
This is the variant of Kontsevich and Soibelman CoHA multiplication \cite{KS11} in case of a quiver with potential. We follow closely the exposition in \cite{D16}.

\textit{Step 1: pull-back along} $(V^\nu)^{\lambda\geq 0}/L_\lambda \to V^\lambda/L_\lambda $

Recall we have the induction diagram of spaces together with algebraic groups they act upon
\[\begin{tikzcd}
	& {(V^\nu)^{\lambda \geq 0}} &&& {P_{\lambda,\nu}} \\
	{V^\lambda} && {V^\nu} & {L_\lambda} && {L_\nu}
	\arrow["{{{\pi_\lambda^\nu}}}"', from=1-2, to=2-1]
	\arrow[hook, from=1-2, to=2-3]
	\arrow[from=1-5, to=2-4]
	\arrow[hook, from=1-5, to=2-6]
\end{tikzcd}\]
with $P_{\lambda,\nu}$-equivariant affine fibration $\pi_\lambda^\nu$ of relative dimension $\dim \pi_\lambda^\nu := \dim (V^\nu)^{\lambda \geq 0} - \dim V^\lambda = \sum_{\alpha\in X^*(T):V^\nu_\alpha\neq 0, \langle\lambda, \alpha\rangle>0}\dim V_\alpha^\nu$, and the right $P_{\lambda,\nu}$-equivariant closed embedding. Here $P_{\lambda,\nu}:=P_\lambda\cap L_\nu$ is a parabolic subgroup of $L_\nu$.

We have the induced affine fibration $((V^\nu)^{\lambda\geq 0},L_\lambda)_N\xrightarrow{p_N} (V^\lambda,L_\lambda)_N$, again of relative dimension $\dim \pi_\lambda^\nu$. It induces an isomorphism $\mbQ_{(V^\lambda,L_\lambda)_N}\to p_{N,*}\mbQ_{((V^\nu)^{\lambda\geq 0},L_\lambda)_N}$ and an isomorphism $\phi_{f_{N,\lambda}}(\mbQ_{(V^\lambda,L_\lambda)_N}\to p_{N,*}\mbQ_{((V^\nu)^{\lambda\geq 0},L_\lambda)_N})$. Applying Verdier duality and using that vanishing cycle commutes with Verdier duality we get an isomorphism $\phi_{f_{N,\lambda}}(p_{N,!}\mbD\mbQ_{((V^\nu)^{\lambda\geq 0},L_\lambda)_N}\to \mbD\mbQ_{(V^\lambda,L_\lambda)_N})$. Using that for $X$ a smooth complex variety $\mbD \mbQ_X \simeq\mbQ_X[2 \dim X]$, we get an isomorphism $\phi_{f_{N,\lambda}}(p_{N,!}\mbQ_{((V^\nu)^{\lambda\geq 0},L_\lambda)_N} \to\mbQ_{(V^\lambda,L_\lambda)_N} [-2 \dim \pi_\lambda^\nu])$.
Using that $\phi_{f_{N,\lambda}}p_{N,!}\mbQ \simeq p_{N,!}\phi_{f_{N,\lambda}\circ p_N}\mbQ$ by (\ref{afffib}),  we have an isomorphism 
\begin{align*}
p_{N,!}\phi_{f_{N,\lambda}\circ p_N}\mbQ_{((V^\nu)^{\lambda \geq 0},L_\lambda)_N}\to \phi_{f_{N,\lambda}}\mbQ_{(V^\lambda,L_\lambda)_N}[-2 \dim \pi_\lambda^\nu].
\end{align*}
Shifting by $[\dim V^\lambda/L_\lambda]$
and taking compactly supported cohomology, we get 
\begin{align*}
    H_c(((V^\nu)^{\lambda\geq 0},L_\lambda)_N, \phi_{f_{N,\lambda}\circ p_N})[\dim V^\lambda/ L_\lambda] \to H_c((V^\lambda,L_\lambda)_N, \phi_{f_{N,\lambda}})[\dim V^\lambda/ L_\lambda - 2\dim \pi_\lambda^\nu].
\end{align*}
Taking colimit and the dual, we arrive to an isomorphism
\begin{align*}
    \alpha:H_{c,L_\lambda}(V^\lambda, \phi_{f_\lambda})^\vee[-\dim V^\lambda/ L_\lambda] \to H_{c, L_\lambda}((V^\nu)^{\lambda \geq 0}, \phi_{f_\lambda^\nu})^\vee[-\dim V^\lambda/L_\lambda -2\dim \pi_\lambda^\nu]
\end{align*}
\textit{Step 2: pullback along $(V^\nu)^{\lambda\geq 0}/L_\lambda \to (V^\nu)^{\lambda\geq 0}/P_{\lambda,\nu}$}
There is an affine fibration   
\begin{align*}
    ((V^\nu)^{\lambda\geq 0}, L_\lambda)_N\xrightarrow{q_N} ((V^\nu)^{\lambda\geq 0}, P_{\lambda,\nu})_N
\end{align*}
of relative dimension $\dim q_\lambda^\nu := \dim \p_{\lambda,\nu}-\dim \el_\lambda = \sum_{\alpha\in X^*(T):\el_{\nu,\alpha}\neq 0, \langle\lambda, \alpha\rangle>0} \dim \el_{\nu, \alpha}$. 
We have a similar isomorphism
\begin{align*}
q_{N,!}\phi_{f_{\lambda, N}^\nu\circ q_N}\mbQ_{((V^\nu)^{\lambda \geq 0},L_\lambda)_N}\to \phi_{f_{\lambda,N}^\nu}\mbQ_{((V^\nu)^{\lambda \geq 0},P_{\lambda,\nu})_N}[-2 \dim q_\lambda^\nu].
\end{align*}
Taking colimits of the shifted duals, we arrive to an isomorphism
\begin{align*}
    \beta: H_{c,P_{\lambda,\nu}}((V^\nu)^{\lambda \geq 0}, \phi_{f_\lambda^\nu})^\vee[-\dim V^\lambda/L_\lambda-2\dim\pi_\lambda^\nu + 2\dim q_\lambda^\nu] \to \\\to H_{c,L_\lambda}((V^\nu)^{\lambda \geq 0}, \phi_{f_\lambda^\nu})^\vee[-\dim V^\lambda/L_\lambda-2\dim \pi^\nu_\lambda] 
\end{align*}
Note that $-\dim V^\lambda/L_\lambda-2\dim\pi_\lambda^\nu + 2\dim q_\lambda^\nu=\dim V^\lambda/L_\lambda -2\dim (V^\nu)^{\lambda\geq0}/P_{\lambda,\nu} = -\dim V^\nu/L_\nu$, where in the last equality we used the assumption (\ref{ass}).

Step 3: \textit{restrict vanishing cycles} 

The closed embedding $((V^\nu)^{\lambda \geq 0}, P_{\lambda,\nu})_N\xrightarrow{i_N} (V^\nu,P_{\lambda,\nu})_N \xrightarrow{f_{N,\nu}}\mbC$ induces a map 
$i_{N}^*\phi_{f_{N,\nu}}\to \phi_{f_{N,\nu}\circ i_N}i_N^*$, and the map $H_{c,P_{\lambda,\nu}}((V^\nu)^{\lambda\geq 0}, ((V^\nu)^{\lambda\geq 0}\to V^\nu)^*\phi_{f_\nu})\to H_{c,P_{\lambda,\nu}}((V^\nu)^{\lambda\geq0}, \phi_{f_\lambda^\nu})$. Taking shifted duals one gets
\begin{align*}
    \epsilon:H_{c,P_{\lambda,\nu}}((V^\nu)^{\lambda\geq0}, \phi_{f_\lambda^\nu})^\vee[-\dim V^\nu/L_\nu]\to H_{c,P_{\lambda,\nu}}((V^\nu)^{\lambda\geq 0}, ((V^\nu)^{\lambda\geq 0}\to V^\nu)^*\phi_{f_\nu})^\vee[-\dim V^\nu/L_\nu]
\end{align*}

Step 4: \textit{extend vanishing cycles}

The closed embedding $((V^\nu)^{\lambda \geq 0}, P_{\lambda,\nu})_N\xrightarrow{i_N} (V^\nu,P_{\lambda,\nu})_N \xrightarrow{f_{N,\nu}}\mbC$ induces a map 
$\phi_{f_{\nu,N}}\to i_{N,*}i_N^* \phi_{f_{\nu,N}}$ and $H_{c,P_{\lambda,\nu}}(V^\nu,\phi_{f_\nu})\to H_{c,P_{\lambda,\nu}}((V^\nu)^{\lambda\geq 0},((V^\nu)^{\lambda \geq 0} \to V^\nu)^*\phi_{f_\nu})$ and taking shifted duals
\begin{align*}
    \zeta: H_{c,P_{\lambda,\nu}}((V^\nu)^{\lambda\geq 0},((V^\nu)^{\lambda \geq 0} \to V^\nu)^*\phi_{f_\nu})^\vee[-\dim V^\nu/L_\nu]\to H_{c,P_{\lambda,\nu}}(V^\nu,\phi_{f_\nu})^\vee[-\dim V^\nu/L_\nu]
\end{align*}

Step 5: \textit{push-forward along} 
$V^\nu/P_{\lambda,\nu}\to V^\nu/L_\nu$

The proper map $(V_\nu, P_{\lambda,\nu})_N\xrightarrow{pr_N} (V_\nu, L_\nu)_N$  induces a map $\phi_{f_{\nu,N}} \to pr_{N,!}\phi_{f_{\nu,N}} pr_N^*$. We get $H_{c,L_\nu}(V^\nu, \phi_{f_\nu})\to H_{c,P_{\lambda,\nu}}(V^\nu,\phi_{f_\nu})$ and its shifted dual
\begin{align*}
    \delta: H_{c,P_{\lambda,\nu}}(V^\nu,\phi_{f_\nu})^\vee[-\dim V^\nu/L_\nu]\to H_{c,L_\nu}(V^\nu,\phi_{f_\nu})^\vee[-\dim V^\nu/L_\nu].
\end{align*}

\textbf{We define the induction map} 
\begin{align*}
    \mathcal{H}_{V,f,\lambda}\xrightarrow{Ind} \mathcal{H}_{V,f,\nu}
\end{align*}
as $\delta\zeta\epsilon\beta^{-1}\alpha$. 

Note that the composition $\zeta\epsilon$ is a push-forward along the proper map $(V^\nu)^{\lambda\geq 0}/P_{\lambda,\nu} \to V^\nu/P_{\lambda,\nu}$.

When $f=0$, $\phi_{f}\mbQ[-1]=\mbQ$ and $\mathcal{H}_{V,f,\lambda}$ is the space of symmetric polynomials $\mbQ[\T]^{W_\lambda}$ with shifted degrees. The induction map is dual to the shuffle map from \cite{H25}
\begin{align}\label{shuf}
    &\mbQ[\T]^{W_\lambda}\xrightarrow{Ind^\lambda_\nu} \mbQ[\T]^{W_\nu}\\
    &f \mapsto \sum_{\sigma\in W_\nu/W_\nu \cap W_{\lambda}}\sigma.(f k_{\lambda,\nu})
\end{align}
where 
\begin{align*}
    k_{\lambda,\nu} = \dfrac{\prod_{\substack{\alpha\in X^*(T)\\ V_\alpha^\nu\neq0, \langle\lambda,\alpha\rangle>0} }\alpha^{\dim V^\nu_\alpha}}{\prod_{\substack{\alpha\in X^*(T)\\\el_{\nu, \alpha}\neq0, \langle\lambda,\alpha\rangle >0}}\alpha ^{\dim \el_{\nu,\alpha}}}
\end{align*}

\begin{remark}
Note that the degrees of the numerator and denominator are $-2\dim \pi_\lambda^\nu$ and $-2 \dim q_\lambda^\nu$, respectively, if we put the degrees of $\alpha\in X^*(T)$ to be $-2$. Note that on one hand the induction map changes the cohomological degree by $-\dim V^\nu/L_\nu + \dim V^\lambda/L_\lambda$ and in this sense preserves it. On the other hand, multiplication by $k_{\lambda,\nu}$ changes the degree by $-2\dim \pi_\lambda^\nu + 2 \dim q_\lambda^\nu = -\dim V^\nu/L_\nu + \dim V^\lambda/L_\lambda$, making no contradictions. 
\end{remark}

\begin{remark}
It is likely that $Ind$ has the same image for $\lambda\sim\lambda^\prime$ but we do not prove it. 
\end{remark}

\begin{example}
    Consider $V=\g$ the adjoint representation, any $\lambda$ and $\nu = 1$, so $L_\lambda \subseteq L_\nu=G$. Then $k_{\lambda,1} =1$ and $Ind^\lambda$ is an operation of symmetrization. 
\end{example}
\begin{example}
    Consider $V=\mbC^3$ the standard representation of $\mathrm{GL}_3(\mbC)$, and choose $\lambda(t)=(t^2,t,t)$ and $\nu(t)=(1,1,1)$. Then 
    \[\begin{tikzcd}
	{\mbQ[\T]^{W_{\lambda}}} & {\mbQ[\T]^{W}} \\
	{\mbQ[z_1,z_2]^{S_2}\otimes \mbQ[z]} & {\mbQ[z_1,z_2,z_3]^{S_3}}
	\arrow["{Ind_\lambda}", from=1-1, to=1-2]
	\arrow["\simeq"', no head, from=1-1, to=2-1]
	\arrow["\simeq"', no head, from=1-2, to=2-2]
	\arrow[from=2-1, to=2-2]
\end{tikzcd}\]
is the shuffle product
\begin{align*}
    (f*g)(z_1,z_2,z_3) = \sum_{\sigma\in S_3/S_2}f(z_{\sigma(1)},z_{\sigma(2)})g(z_{\sigma(3)}) \dfrac{z_{\sigma(1)}z_{\sigma(2)}z_{\sigma(3)}}{(z_{\sigma(1)}-z_{\sigma(2)})(z_{\sigma(1)}-z_{\sigma(3)})}
\end{align*}
\end{example}

\subsection{Induction for cotangent representations}\label{subsec:indcotrep}

Let $V$ be a finite dimensional representation of a complex reductive group $G$ with Lie algebra $\g$. Let $\mu_V:T^*V\to \g^*$ be a $G$-equivariant moment map. Consider a $G$-invariant function $f:T^*V\times \g\to \mbC, (x,x^*,\xi)\mapsto \mu_V(x,x^*)(\xi)$. Its critical locus is $\mathrm{Crit}(f)=\{(x,x^*,\xi): \mu_V(x,x^*)=0, \xi\in \g_x \cap \g_{x^{*}}\}\subset \mu_V^{-1}(0)\times \g\subset f^{-1}(0)$, where $\g_x$ is the Lie algebra of the stabilizer of $x\in V$ and $\g_{x^{*}}$ is the Lie algebra of the stabilizer of $x^*\in V^*$. We have $\phi_f$ is supported at $\Crit(f)$.

  For $\lambda\in X_*(T)$
denote by $d_\lambda=\dim ~T^*V^\lambda$ and $l_\lambda=\dim~\el_\lambda$. We denote by $\phi_\lambda:=\phi_{f_\lambda}$ for $f_\lambda$ the restriction of $f$ to the $\lambda$-fixed locus $T^*V^\lambda\times\el_\lambda$.

Recall the order relation (\ref{ord}): $\lambda\preceq \nu \iff$  $V^\lambda \subseteq V^\nu$ and $\el_\lambda\subseteq \el_\nu$ and the last conditions is equivalent to $L_\lambda \subseteq L_\nu$ since we work over an algebraically closed field. Whenever $\lambda \preceq \nu$, we have a morphism in $D^b_c(\pt,\mbQ)$ 
\begin{align}\label{indcot}
H_{L_\lambda}^\BM(\mu_\lambda^{-1}(0),\mbQ)[d_\lambda +2l_\lambda]\xrightarrow{Ind_\nu^\lambda}H_{L_\nu}^\BM(\mu_\nu^{-1}(0),\mbQ)[d_\nu +2l_\nu],
\end{align}
called the \textit{Hall induction} from $\lambda$ to $\nu$. We define it as (\text{dim.red.})$\circ Ind^\lambda_\nu\circ (\text{dim.red.})^{-1}$   from the diagram 
\[\begin{tikzcd}
	{ H^\BM_{d_\lambda+2l_\lambda-i,  L_{\lambda}}(\mu_{\lambda}^{-1}(0),\mbQ)} & {H^\BM_{L_{\nu}}(\mu_{\nu}^{-1}(0),\mbQ)} \\
	{ H^{i-d_\lambda-2l_\lambda}_{c,  L_{\lambda}}(\mu_{\lambda}^{-1}(0),\mbQ)^\vee} & { H^{i-d_\nu-2l_\nu}_{c,  L_{\nu}}(\mu_{\nu}^{-1}(0),\mbQ)^\vee} \\
	{H_{c,L_\lambda}^{i-d_\lambda}(T^*V^\lambda\times\el_\lambda, \phi_\lambda)^\vee} & {H_{c,L_\nu}^{i-d_\nu}(T^*V^\nu\times\el_\nu, \phi_\nu)^\vee}
	\arrow["{Ind_\nu^\lambda}", from=1-1, to=1-2]
	\arrow["{=}"', no head, from=1-1, to=2-1]
	\arrow["{=}"', no head, from=1-2, to=2-2]
	\arrow["{\text{dim.red.}\simeq}"', no head, from=2-1, to=3-1]
	\arrow["{\text{dim.red.}\simeq}"', no head, from=2-2, to=3-2]
	\arrow["{Ind_\nu^\lambda}", from=3-1, to=3-2]
\end{tikzcd}\]
where \text{dim.red} stands for dimensional reduction isomorphism from Example (\ref{eqdimred}). The map below is defined as a special case of the KS critical Hall induction (\ref{KScrit})
\begin{align}\label{T*VKS}
\mathcal{H}_{T^*V\times \g,f,\lambda}\xrightarrow{Ind^\lambda_\nu}\mathcal{H}_{T^*V\times \g,f,\nu}
\end{align}
for $f(x,x^*,\xi)=\mu_V(x,x^*)(\xi)$.

\subsection{Associativity}
 \begin{lemma}
 For cocharacters $\lambda, \mu, \nu \in X_*(T)$ such that $\lambda \preceq \mu \preceq \nu$ we have
    \begin{align*}
    Ind^{\mu}_{\nu} \circ Ind^{\lambda}_{\mu} = Ind^{\lambda}_{\nu} 
    \end{align*}   
\end{lemma} 
\begin{proof}
    This is proven the same way as in \cite{KS11}.
\end{proof}

\section{Torsion freeness}\label{TF}
Let $V$ be a representation of reductive group $G$ and suppose $T_s$ is an auxiliary torus acting on $T^*V\times \g$ and verify the following \textbf{assumptions on }$T_s$ :

\begin{itemize}
    \item  $T_s$ acts on $T^*V\times \g$, preserves $\mu_V^{-1}(0)\subset T^*V\times \g$, and commutes  with the action of $G$,
    \item  the function $f:T^*V\times \g\to\mbC, ~(x,x^*,\xi)\mapsto \mu_V(x,x^*)(\xi)$ is $T_s$-invariant,
    
    \item  $T_s$ contains two 1-dimensional subtori $\mbC_1^*, \mbC_2^*$ acting on $T^*V\times \g$ with weights $(1,-1,0), (1,0,-1)$, respectively
    \end{itemize}

 Denote by $\mu_V^{-1}(0)\xrightarrow{i} T^*V\times \g$ the $G\times T_s$-equivariant embedding. In this section we want to show the pushforward \begin{align}\label{locBM}
    H^{\BM}_{G\times T_s}(\mu_{V}^{-1}(0), \mbQ) \xrightarrow[]{i_*} H^{\BM}_{G\times T_s}(T^*V\times \g, \mbQ)\simeq H_{G\times T_s}
\end{align}
is an embedding under the above assumptions on $T_s$. Denote by $\pt=\mu_{V}^{-1}(0)^{\mbC^*_1}\xrightarrow{j} \mu_{V}^{-1}(0)$ the embedding of the $\mbC^*_1$-fixed locus, which is a point. We have a commutative diagram 
\[\begin{tikzcd}
	{\mu_V^{-1}(0) } & {T^*V\times \g} \\
	{\pt=\mu_{V}^{-1}(0)^{\mbC^*_1}}
	\arrow["i", hook, from=1-1, to=1-2]
	\arrow["j", hook, from=2-1, to=1-1]
	\arrow["ij", hook, from=2-1, to=1-2]
\end{tikzcd}\]
giving a commutative diagram of vector spaces
\[\begin{tikzcd}
	{H^{\BM}_{G\times T_s}(\mu_{V}^{-1}(0), \mbQ)} & {H^{\BM}_{G\times T_s}(T^*V\times \g, \mbQ)} \\
	{H_{G\times T_s}}
	\arrow["{i_*}", from=1-1, to=1-2]
	\arrow["{j^*}"', from=1-1, to=2-1]
	\arrow["{\simeq(ij)^*}", from=1-2, to=2-1]
\end{tikzcd}\]
The space $T^*V \times \g$ is $G\times T_s$-contractible to the fixed point, thus $(ij)^*$ is an isomorphism.

\begin{remark}\label{tfae}
    In this situation the following is equivalent:  
\begin{itemize}
    \item [$\cdot$] $i^*$ is an embedding,
    \item [$\cdot$] $j^*$ is an embedding,
    \item [$\cdot$] the $H_{G\times T_s}$-module $H^{\BM}_{G\times T_s}(\mu_{V}^{-1}(0), \mbQ)$ is torsion free. 
\end{itemize}
\end{remark}

For future use we state the version of Atiyah-Bott localization theorem 
\begin{theorem}\cite[Theorem 6.2 (3)]{GKM}
    Let $X$ be a complex algebraic variety with an action of a torus $K$. Let 
    $L\subseteq K$ be a subtorus.
    Denote by $I=\ker(H_K\to H_L)$ the prime ideal of functions on the Lie algebra of $K$ that vanish on the Lie algebra of $L$, and by $S= H_K\backslash I$ the its complement. 
     Then the localized restriction map to the $L$ fixed locus $X^L$ is an isomorphism
     \begin{align*}
         H_K(X,\mbQ)[S^{-1}]\xrightarrow{\sim} H_K(X^L,\mbQ)[S^{-1}] 
     \end{align*}
\end{theorem}
We illustrate the theorem by an example.

\textbf{Example}
   
    Let the torus $(\mbC^*)^2$ acts on $\mathbb{P}^1$ by 
    \begin{align*}
        [x:y]\mapsto [t_1 x:t_2 y]
    \end{align*} with fixed locus, consisting of two points $[1:0]$ and $[0:1]$. Consider the subtorus $\mbC^* \subset (\mbC^*)^2, z\mapsto (z,z^{-1})$ with the same fixed locus. The ideal $I=\ker(\mbQ[z,w]\to \mbQ[t], z\mapsto t, w\mapsto -t)$ is generated by $z+w$. We have
    \begin{align*}
        H_{(\mbC^*)^2}(\mathbb{P}^1,\mbQ)\simeq \mbQ[z,w,u]/(u-z)(u-w)
    \end{align*}
    where $u=c_1^{(\mbC^*)^2}(\mathcal{O}(1))\in H^2_{(\mbC^*)^2}(\mathbb{P}^1,\mbQ)$ and $z,w$ are the first Chern classes of $\mathcal{O}(1)$ over  $B(\mbC^*)^2=\mbC\mathbb{P}^\infty\times \mbC\mathbb{P}^\infty$. The cohomology of the fixed locus is $$H_{(\mbC^*)^2}(\pt\coprod\pt,\mbQ)=\mbQ[z,w]^{\oplus 2}.
    $$
    
    Inverting the function $z-w$, which is in the complement of $I$, ideals $(u-z),(u-w)\subset \mbQ[z,w,u]$ become coprime in the localization. By Chinese reminder theorem,
    $$ H_{(\mbC^*)^2}(\mathbb{P}^1,\mbQ)[\frac{1}{z-w}] \simeq \mbQ[z,w,\frac{1}{z-w}]^{\oplus 2}.$$
    The localized restriction map yields an isomorphism
    \begin{align*}
        H_{(\mbC^*)^2}(\mathbb{P}^1,\mbQ)[\frac{1}{z-w}] \xrightarrow{i^*} \mbQ[z,w][\frac{1}{z-w}]\oplus \mbQ [z,w][\frac{1}{z-w}]
    \end{align*}
    $$u\mapsto (z,w).$$

 Note that for any linear action of $T_s$ on $\g$ the nilpotent cone $\mathcal{N}\subset \g$ and nilpotent orbits $\mbO$ are invariant under the action of $G\times T_s$.

\subsection{The statement}
In this section we prove

\begin{theorem}\label{tfree}
    Under the above assumptions on $T_s$, the $H_{G\times T_s}$-module $H^{\BM}_{G\times T_s}(\mu_{V}^{-1}(0), \mbQ)$ is torsion free.
\end{theorem}


\begin{proof}

Recall the isomorphism of $H_{G\times T_s}$ with $G\times T_s$-invariant functions on its Lie algebra $H_{G\times T_s}\simeq \mbQ[\g \times \T_s]^{G\times T_s}$. Denote by $I_1\subset H_{G\times T_s}$ the prime ideal of functions vanishing on the Lie algebra of $\mbC^*_1$. In other words, $I_1=\ker (H_{G\times T_s}\to \mathrm{Lie}(\mbC^*_1))$.  Denote by $S=H_{G\times T_s}\backslash I_1$ the complement to $I_1$, its elements form a multiplicative system.  The localized restriction map (we invert elements from $S$)  
\begin{align}\label{AB}
H^{\BM}_{G\times T_s}(\mu_{V}^{-1}(0), \mbQ)[S^{-1}]
    \xrightarrow{j^*}
H_{G\times T_s}[S^{-1}]
\end{align}
is an isomorphism by the variant of Atiyah-Bott localization theorem \cite[Theorem 6.2 (3)]{GKM} and the isomorphism $$H^{\BM}_{G\times T_s}(\mu_V^{-1}(0),\mbQ))\simeq H^{\BM}_{D\times T_s}(\mu_V^{-1}(0),\mbQ))^W$$ where $D\subset G$ is the maximal torus and $W=N_G(D)/D$ is the Weyl group. The RHS of \ref{AB} is torsion free, hence $H^{\BM}_{G\times T_s}(\mu_{V}^{-1}(0),\mbQ)$ is torsion free over $H_{G\times T_s}$ if and only if it is torsion free over $S$.

Denote by $k_1 = H_{\mbC^*_1}, ~ k_2 = H_{\mbC^*_2}$ and by $K_1,~ K_2$ their fields of fraction. We have the variant of \cite[Theorem 9.6]{D22}
\begin{lemma}\label{eemb}
    The module $H^{\BM}_{G\times T_s}(\mu_{V}^{-1}(0),\mbQ)$ is free as $k_2$-module. As a consequence, the natural map 
    \begin{align*}
    H^{\BM}_{G\times T_s}(\mu_{V}^{-1}(0),\mbQ) \to H^{\BM}_{G\times T_s}(\mu_{V}^{-1}(0),\mbQ)\otimes_{k_2} K_2
\end{align*}
is injective.
\end{lemma}
 \begin{proof}
The argument goes along the lines of the proof of \cite[Theorem 9.6]{D22} coupled with the purity of $H^\BM_G(\mu_V^{-1}(0),\mbQ)$ \cite[Corollary 1.11]{H25a}. Namely, pick a splitting $T_s \simeq T^\prime \times T_\chi$. Then $\T\simeq \T^\prime \oplus \T_\chi$ and $H_{T_s}\simeq H_{T^\prime} \otimes H_{T_\chi}$. Consider a $L_\lambda\times T_s$-variety  $$V_{N}=\fr(n,N)\times \Hom(\mbC^N,\T)$$ where $L_\lambda$ acts on $\fr(n,N)$ via the fixed embedding $L_\lambda \subset G \subset GL_n(\mbC)$ by changing tuples of $n$ linearly independent vectors in $\mbC^N$ and $T_s$ acts on $\Hom(\mbC^N,\T)$ by scaling in the image with weights 1. The group $L_\lambda \times T_s$ acts freely on the open subset $U_{N}\subset V_{N}$, defined by asking the linear maps in $\Hom(\mbC^N,\T)$
to be surjective. Denote by 
$L_\lambda^\tau:= L_\lambda \times T_s$ and $L_\lambda^\prime:= L_\lambda \times T^\prime$.
Define the smooth varieties 
\begin{align*}
 Y_{\lambda,N}:= (T^*V^\lambda\times \el_\lambda) \times^{L_\lambda^\tau} U_{N}\\
 Y_{\lambda,N}^\prime:= (T^*V^\lambda\times \el_\lambda) \times^{L_\lambda^\prime} U_{N}
\end{align*}
and denote by $f_{\lambda,N}$ and $f_{\lambda,N}^\prime$ the induced functions on them. 
 We have a morphism $$v_{\lambda,N}:Y_{\lambda,N}\to U_{N}/L_\lambda^\tau \to \Hom^{\mathrm{surj}}(\mbC^N,\T_\chi)/T_\chi:=S_{\chi,N}.$$
 The map $v_{\lambda,N}$ is locally trivial with fiber $Y_{\lambda,N}^\prime$. Then the sheaves $R^q v_{\lambda,N,*}\mbD\phi_{f_{\lambda,N}}\mbQ$ are local systems with fiber $H^q(Y_{\lambda,N}^\prime, i^*\mbD\phi_{f^\prime_{\lambda,N}}\mbQ)$, where $i$ is an embedding $Y_{\lambda,N}^\prime \to Y_{\lambda,N}$ of the fiber, denote by $c$ its codimension. By \cite[Lemma 9.5]{D22} the target manifold $S_{\chi,N}$ is simply connected, then the local systems are trivial. 

The Leray spectral sequence associated to the morphism $v_{\lambda,N}$ has on its second page 
$$E^{p,q}_2=H^p(S_{\chi,N}, R^q v_{\lambda,N,*}\mbD\phi_{f_{\lambda,N}}\mbQ) \simeq H^p(S_{\chi,N},\mbQ)\otimes H^q(Y_{\lambda,N}^\prime,i^*\mbD\phi_{f_{\lambda,N}}\mbQ)$$
and converges to $H^{p+q}(Y_{\lambda,N},\mbD\phi_{f_{\lambda,N}}\mbQ)=H^\BM_{-p-q}(Y_{\lambda,N},\phi_{f_{\lambda,N}}\mbQ)$.

We have $i^*\mbD\phi_{f_{\lambda,N}}\mbQ=\mbD i^! \phi_{f_{\lambda,N}}\mbQ = \mbD i^* \phi_{f_{\lambda,N}}\mbQ[-2c],$ 
then 
\begin{align}\label{E2}
    E^{p,q}_2 = H^p(S_{\chi,N},\mbQ)\otimes H^\BM_{-q+2c}(Y_{\lambda,N}^\prime, \phi_{f_{\lambda,N}^\prime}\mbQ)).
\end{align}
 By \cite[Corollary 1.11]{H25a} and \cite[Lemma 9.5]{D22} the RHS of (\ref{E2}) is pure. Then the spectral sequence $E^{\bullet,\bullet}_\bullet$ degenerates on its second page. 

    The rest of the argument is similar to that in \cite[Theorem 9.6]{D22}.
    \end{proof}

Combining dimensional reduction isomorphisms from  Examples \ref{ex1dimred} and \ref{ex2dimred} one gets
\begin{align*}
    H^{\BM}_{G\times T_s}(\mu_{V}^{-1}(0),\mbQ) \simeq H^{\BM}_{G\times T_s}(\{(x,a)\in V\times \g : a.x = 0\},\mbQ).
\end{align*}

By the argument in \cite[Proposition 5.2]{SV20}
the localized pushforward map
\begin{align*}
    H^{\BM}_{G\times T_s}(\{(x,a)\in V\times \mathcal{N} : a.x = 0\},\mbQ)\otimes_{k_2}K_2\to
    H^{\BM}_{G\times T_s}(\{(x,a)\in V\times \g : a.x = 0\},\mbQ)\otimes_{k_2}K_2
\end{align*}
is an isomorphism.

The space $\widehat{\mathcal{N}}:=\{(x, a) \in V\times \mathcal{N} : a.x = 0\}$ is stratified 
\begin{align*}
    \widehat{\mathcal{N}} = \coprod_\lambda \{(x, a) \in V\times \mbO_\lambda : a.x = 0\}.
\end{align*}
\begin{remark}
The $G\times T_s$-equivariant projection $\widehat{\mathcal{N}}\to \mathcal{N}$ restricted to each stratum is an affine fibration, inducing an isomorphism (up to shift) in Borel-Moore homology
 \begin{align*}
     H^\BM_{G\times T_s}(\{(x, a) \in V\times \mbO_\lambda : a.x = 0\},\mbQ)\to H^\BM_{G\times T_s}(\mbO_\lambda,\mbQ).
 \end{align*}
\end{remark}

To show $H^\BM_{G\times T_s}(\widehat{\mathcal{N}},\mbQ)$ has no $S$-torsion we show this for the space $H^\BM_{G\times T_s}({\mathcal{N}},\mbQ)$ and from the argument and by the remark above this will be sufficient.

Given a complex algebraic variety $X$ with an action of a complex algebraic group $G$ and a closed $G$-invariant subset $Z\subset X$, denote by 
\[\begin{tikzcd}
	Z & X & {U=X-Z}
	\arrow["i", hook, from=1-1, to=1-2]
	\arrow["j"', hook', from=1-3, to=1-2]
\end{tikzcd}\]
the closed an open embeddings leading a long exact sequence 
\begin{align*}
    \dots \to H^{\BM}_{i,G}(Z,\mbQ)\to H_{i,G}^\BM(X,\mbQ)\to H_{i,G}^\BM(U,\mbQ)\xrightarrow{\delta}H^{\BM}_{i-1,G}(Z,\mbQ)\to \dots 
\end{align*}

Moreover, this is a sequence of mixed Hodge structures. 

Having a (finite)stratification of nilpotent cone by nilpotent orbits $\mathcal{N}=\coprod_{\lambda}\mbO_{\lambda}$, we will, by induction   on dimensions of orbits, successively split off the strata of increasing dimension. 



Namely, we have a filtration of nilpotent cone
\begin{align*}
    \cdots\subset X_{k-1}\subset X_k\subset\cdots\subset \mathcal{N}
\end{align*}

by closed invariant subsets
\begin{align*}
X_k=\bigcup_{\dim \mbO_\lambda \leq a_k}\mbO_\lambda
\end{align*}
with 
\begin{align*}
    Z_k=X_{k}-X_{k-1} = \coprod_{\dim \mbO_\lambda=a_k} \mbO_\lambda.
\end{align*}
for certain numbers $0=a_0<a_1<a_2<\dots$, assuming $X_{-1}=\emptyset$.

We have an open-closed pair 
\[\begin{tikzcd}
	{X_{k-1}} & {X_k} & {Z_k}
	\arrow["i", hook, from=1-1, to=1-2]
	\arrow["j"', hook', from=1-3, to=1-2]
\end{tikzcd}\]
and it leads to a long exact sequence of mixed Hodge structures
\begin{align*}
    \dots \xrightarrow{\delta} H^{\BM}_{i,\G}(X_{k-1},\mbQ)\to H_{i,\G}^\BM(X_k,\mbQ)\to H_{i,\G}^\BM(Z_k,\mbQ)\xrightarrow{\delta}H^{\BM}_{i-1,\G}(X_{k-1},\mbQ)\to \dots. 
\end{align*}
To show its splitting, we show the Hodge structures $H_{i,\G}^\BM(Z_k,\mbQ),~H^{\BM}_{i-1,\G}(X_{k-1},\mbQ)$ are pure but of different weights. Hence the connecting homomorphism $\delta=0$.

\begin{lemma}
    We have  $H^{\BM}_{2i+1, \G}(\mbO_{\lambda}, \mbQ)=0$ and the mixed Hodge structure $$H^{\BM}_{2i, \G}(\mbO_{\lambda}, \mbQ)\otimes \mbC$$ is pure.
\end{lemma}
\begin{proof}
    One has $H^{\BM}_{i, G\times T_s}(\mbO_{\lambda},\mbQ)= H^\BM_{i-2\dim G\times T_s}(G\times T_s \backslash G\times T_s /Stab_\lambda, \mbQ) =H^{-i+2\dim G\times T_s}(BStab_\lambda,\mbQ)$ and the result follows from \cite[Théorème 9.1.1]{DIII}. Here we recall the argument.
   For a complex linear algebraic group $G$ let $G^0$ be the connected component of 1 in $G$ and $T\subset G^0$ be the maximal torus with $W$ the Weyl group. Then $H(BG,\mbQ)=H(BG^0,\mbQ)^{G/G^0}$ and $H(BG^0,\mbQ)\simeq H(BT,\mbQ)^W$.  One has $T\simeq \mbG_m^n$ and by Künneth it is enough to consider $T=\mbG_m$. 
    Then $H^{2i}(BT,\mbQ)=H^{2i}(\mbC\mathbb{P}^\infty,\mbQ)=\mbQ$ is generated by a class Poincaré dual to algebraic cycle $[\mbC\mathbb{P}^i]$ and so is pure of type $(i,i)$, and is zero in odd degrees. Taking invariants wrt finite groups $G/G^0$ and $W$ does not effect the Hodge type. We conclude $H^{\BM}_{2i, G\times T_s}(\mbO_{\lambda}, \mbQ)\otimes \mbC$ is pure of type $(-i+\dim G\times T_s,-i+\dim G\times T_s)$.
\end{proof}

\begin{corollary}
    We have  $H^{\BM}_{2i+1, \G}(Z_k, \mbQ)=0$ and the mixed Hodge structure $$H^{\BM}_{2i, \G}(Z_k, \mbQ)\otimes \mbC$$ is pure.
\end{corollary}
\begin{proof}
    Borel-Moore homology are additive on disjoint unions and direct sum of pure Hodge structures is pure.
\end{proof}

\begin{lemma}
    Mixed Hodge structure $H^{\BM}_{i,\G}(X_{k},\mbQ)$ is pure.
\end{lemma}
\begin{proof}
    By induction on $k$. We have $X_0$ the nilpotent orbit of minimal dimension, a point $0\in \mathcal{N}$. We have $H^{\BM}_{i,\G}(X_0,\mbQ)=H^{\BM}_{i-2\dim \G}(B\G,\mbQ)$ is pure. By induction, assume $H^{\BM}_{i,\G}(X_{k-1},\mbQ)$ is pure. We use the fact that if $A$ and $B$ are pure Hodge structures of different weights, then $\Hom_{MHS}(A,B)=0$. Then homomorphisms $\delta$ vanish. When $i$ is odd we get isomorphisms $H^{\BM}_{i,\G}(X_{k-1},\mbQ)\to H^{\BM}_{i,\G}(X_{k},\mbQ)$ and then the latter is pure. When $i$ is even we have a short exact sequence
    \begin{align*}
    0\to H^{\BM}_{i,\G}(X_{k-1},\mbQ)\to H_{i,\G}^\BM(X_k,\mbQ)\to H_{i,\G}^\BM(Z_k,\mbQ)\to 0. 
\end{align*}
An extension of pure structures is pure.
\end{proof}

 The purity imply the splitting of the long exact sequences into direct sums of short exact sequences, inducing a filtration on $H^{\BM}_{i,\G}
 (\mathcal{N}, \mbQ)$ 
 \begin{align*}
     \cdots\subset H^{\BM}_{i,\G}  (X_{k-1}, \mbQ)\subset H^{\BM}_{i,\G}  (X_k, \mbQ)\subset\dots \subset H^{\BM}_{i,\G}
 (\mathcal{N}, \mbQ)
 \end{align*}
 
 whose associated graded is $\bigoplus_{\lambda}H^{\BM}_{\G}(\mbO_{\lambda}, \mbQ)$. 
The morphism $$H^{\BM}_{\G}(\mathcal{N}, \mbQ) \to \bigoplus_{\lambda}H^{\BM}_{\G}(\mbO_{\lambda}, \mbQ)$$ is injective. The same argument applied to $\widehat{\mathcal{N}}$ allows to show the morphism $$H^{\BM}_{\G}(\widehat{\mathcal{N}}, \mbQ) \to \bigoplus_{\lambda}H^{\BM}_{\G}(\widehat{\mbO}_{\lambda}, \mbQ)$$ is injective. Since $H^{\BM}_{\G}(\widehat{\mbO}_{\lambda}, \mbQ)\to H^{\BM}_{\G}({\mbO}_{\lambda}, \mbQ)$ are isomorphisms, it is enough
to check that each $H^{\BM}_{\G}(\mbO_{\lambda}, \mbQ)\simeq H_{Stab_\lambda}$ has no $S$- torsion.

Here the $H_{G\times T_s}$-module structure on $H_{Stab_\lambda}$ comes from the inclusion $Stab_\lambda\subset G\times T_s$.
 Recall that the subtorus $\mbC^*_1\subset T_s$ acts with weights $(1,-1,0)$ on $V\times V^*\times\g$. In particular, it does not act on $\g$. That means that $\mbC^*_1$ sits inside each stabilizer $Stab_{\lambda}$. That means that the kernel 
\begin{align*}
    \ker(H_{G\times T_s}\to H_{Stab_{\lambda}}),
\end{align*}
consisting of functions on $\g\times \T_s$ restricting by zero to $\mathrm{Lie}(Stab_{\lambda})$, restricts by zero on $\mathrm{Lie}(\mbC^*_1)$ as well. That means that $\ker(H_{G\times T_s}\to H_{Stab_{\lambda}}) \subset I_1$. For the rest of the proof we adopt the notations: the ring $R:=H_{G\times T_s}$, the $R$-module $M:=H_{Stab_\lambda}$, the ideal $K:=\ker(H_{G\times T_s}\to H_{Stab_{\lambda}})\subset R$.
As before $I_1\subset R$ the ideal of functions vanishing on the Lie algebra of $\mbC^*_1$, and $S=R\backslash I_1$ is its complement.

To show $M$ has no $S$-torsion we need to show $Ann(m)\subset K$ for all $m\in M$.
\begin{lemma}
    Let $G$ be a complex connected reductive group and $H\subset G$ a subgroup. The image of the induced restriction map 
    \begin{align*}
        H(BG,\mbQ)\xrightarrow{res} H(BH,\mbQ)
    \end{align*}
    is a domain. 
\end{lemma}
\begin{proof}
    We have a short exact sequence of groups $1\to H_0\to H\to \pi_0(H)\to 1$ where $H_0$ is a connected component of 1 and $\pi_0(H)$ is a finite group of components. We have $H(BH,\mbQ) = H(BH_0,\mbQ)^{\pi_0(H)}$ is the subspace of $\pi_0(H)$-invariants. Recall an isomorphism with Weyl group invariant functions  $H(BH_0,\mbQ)\simeq \mbQ[\mathfrak{h}_0]^{W_0}$, where $\mathfrak{h}_0$ is a Lie algebra of a maximal torus in $H_0$ and $W_0$ is a finite Weyl group. By the theorem of Chevalley, invariant rings of finite reflexion groups are polynomial rings, hence are domains. A subring $H(BH,\mbQ)\subset H(BH_0,\mbQ)$ of $\pi_0(H)$-invariants is hence a domain as well. 
\end{proof}
Applying for $H=Stab_\lambda$, we get $R/K\simeq \im(res)$ is a domain and $R/K$ acts on $M$ via the same restriction map, since elements from $K$ act by zero. Suppose $Ann(m) \ni r\neq 0$ for some $0\neq m\in M$, that is $rm =0$. Denote by $[r]$ its image in $R/K$. Then $[r]m=0$. Before we showed $M$ has no $I_1$-torsion. Since $K\subset I_1$, $M$ has no $K$-torsion. Together with Lemma above we see that $[r]m=0$ implies $[r]=0$ that is $r\in K$. QED 
\end{proof}

\begin{corollary}\label{corr}
    $H_{G\times T_s}^\BM(\mu_V^{-1}(0),\mbQ)$ is concentrated in even homological degrees.
\end{corollary}

We illustrate the above theorem by a simple example.

\textbf{Example}

    Let $G=GL_2(\mbC)$, $T_s = (\mbC^*)^2$ and $V=\Hom(\mbC^2,\mbC^2)$. Let the element $(g, t=\begin{pmatrix} t_1 & 0 \\ 0 & t_2 \end{pmatrix}) \in G\times T_s$ act on $(x,x^*,a)\in V\times V^*\times \g \simeq \gl_2^3$ by
\begin{align*}
(g, t). (x,x^*,a)=(gxg^{-1}t_2,gx^*g^{-1}t_1^{-1},gag^{-1}t_1/t_2).
\end{align*}
Then the assumptions on $T_s$ are verified. The moment map writes as 
\[\begin{tikzcd}
	{T^*V\simeq V\times V} & {\gl_2^*\simeq \gl_2} & {(x,x^*)} & {[x,x^*]}
	\arrow["{\mu_V}", from=1-1, to=1-2]
	\arrow[maps to, from=1-3, to=1-4]
\end{tikzcd}\]
and the function $f:T^*V\times \g \to \mbC$ sends $(x,x^*,\xi)$ to $\Tr([x,x^*]\xi)$.
The two subtori are
 
\begin{align*}
\mbC^*_1= \{\begin{pmatrix} t & 0 \\ 0 & t \end{pmatrix}\}, ~\mbC^*_2 = \{\begin{pmatrix} 1 & 0 \\ 0 & t \end{pmatrix}\}. 
\end{align*}
 The nilpotent cone  
 
 $\mathcal{N}_{\mathfrak{gl}_2} = \{\begin{pmatrix} \alpha & \beta \\ \gamma & \delta \end{pmatrix}: \alpha \delta-\beta \gamma=0, ~ \alpha+\delta=)\}\subset \mbC^4
 $ is stratified by two orbits $\mbO_{(1,1)}=pt$, the orbit of $0$ and $\mbO_{(2)}$, the complement to the vertex, the orbit of $\begin{pmatrix}
     0 & 1\\
     0 & 0
 \end{pmatrix}$ with $\mbO_{(1,1)}\subset\bar{\mbO}_{(2)}$. The stabilizers are $Stab_{(1,1)}=GL_2\times T_s$ and $Stab_{(2)}=\{\begin{pmatrix} a & b \\ 0 & d \end{pmatrix}, (t_1,t_2) : \frac{t_1}{t_2}= \frac{d}{a}\}\simeq \mbG_m^3 \times \mbG_a$(omitting group structure). 
 We have 
 \begin{align*}
 &H_{GL_2(\mbC)\times T_s}=H(BGL_2(\mbC)\times BT_s)=\mbQ[c_1,c_2,\xi,\eta],\\
 &H_{Stab_{(2)}} = H(B\mbC^*\times B\mbC^* \times B\mbC^*)=\mbQ[A,B,\eta], 
 \end{align*}
Here $A,B$ are the Chern roots of the tautological bundle over $BGL_2(\mbC)$, and $c_1 = A+B, ~ c_2 = AB$ are its Chern classes. Also $\xi,~\eta$ are Chern roots of the tautological bundle over $BT_s$. From another perspective, $A$ and $B$ generate the ring of functions on a Lie algebra of diagonal torus in $GL_2(\mbC)$. From defining equation of $Stab_{(2)}$ we know $A+\xi = B+\eta$. We have $I_1 = (c_1,c_2,\xi-\eta)$. We compute
\[\begin{tikzcd}
	{H_{GL_2(\mbC)\times T_s}} & {H_{Stab_{(2)}}}
	\arrow[from=1-1, to=1-2]
\end{tikzcd}\]
\begin{align*}
    c_1 \mapsto A+B, ~c_2\mapsto AB, ~\xi  \mapsto B-A+\eta,~ \eta\mapsto \eta.
\end{align*}
Thus we get the presentation 
\begin{align*}
H_{Stab_{(2)}}=\mbQ[A,B,\eta]\simeq \mbQ[c_1,c_2,\xi, \eta]/(c_1^2-4c_2-(\xi-\eta)^2).
\end{align*}
One checks it has no torsion over the complement to $I_1$.

\section{Wheel conditions}\label{wheelC}

In the previous section we saw that  under certain assumptions on the torus $T_s$ the
pushforward \begin{align*}
    H^{\BM}_{G\times T_s}(\mu_{V}^{-1}(0), \mbQ) \xrightarrow[]{i_*} H^{\BM}_{G\times T_s}(T^*V\times \g, \mbQ)\simeq H_{G\times T_s}
\end{align*}
or, equivalently, the restriction to the fixed point
\begin{align*}
    H^{\BM}_{G\times T_s}(\mu_{V}^{-1}(0), \mbQ) \xrightarrow[]{j^*} H_{G\times T_s}
\end{align*}
are embeddings. These maps are compatible with Hall induction on the source and on the target and the knowledge of the image of restriction map gives the realization of the Hall induction on $H^{\BM}_{G\times T_s}(\mu_{V}^{-1}(0), \mbQ)$ in terms of symmetric polynomials.
In this section we study the K-theoretic version of the restriction map. 

\subsection{KHA of a one-loop quiver}
Our main source of inspiration was the paper \cite{Z} of Y.~Zhao who studied the image of a $K$-theoretic Hall algebra(KHA) of surfaces to a shuffle algebra. In particular, for $\mbA^2$ he considered the stack of 0-dimensional(hence of finite length) coherent sheaves on $\mbA^2$
\begin{align*}
    \mathrm{Coh}(\mbA^2)=\coprod \mathrm{Coh}_n(\mbA^2),
\end{align*}
where the component $\mathrm{Coh}_n(\mbA^2)$ of  length $n\geq 1$ sheaves is isomorphic to the stack of pairs 
\begin{align*}
    Comm_n:= \{(x,y)\in \gl_n(\mbC)^2: [x,y]=0\}
\end{align*}
of $n \times n$ commuting matrices, up to simultaneous conjugation 
\begin{align*}
    \mathrm{Coh}_n(\mbA^2) \simeq Comm_n/GL_n(\mbC).
\end{align*}

Let the torus $T_s=(\mbC^*)^2$ acts on $\gl_n^2$ by $(x,y)\mapsto (qx,q^\prime y)$. Clearly this action lifts to the action on $\mathrm{Coh}(\mbA^2)$. Denote by $\pt\xrightarrow{i_0} \gl_n^2 \xleftarrow{p}Comm_n$ the $GL_n(\mbC)\times T_s$-equivariant inclusions of a fixed point, the origin, and of variety $Comm_n$.

The K-theoretic Hall algebra(deformed by $T_s$) of abelian category of 0-dimensional coherent sheaves on $\mbA^2$  is a structure of an associative algebra on 
\begin{align*}
    K_{(\mbC^*)^2}(\mathrm{Coh}(\mbA^2))=\bigoplus_{n\geq 1}K_{(\mbC^*)^2}(\mathrm{Coh}_n(\mbA^2))=\bigoplus_{n\geq1} K_{GL_n(\mbC)\times (\mbC^*)^2}(Comm_n)
\end{align*}
where the product comes via the stack of extensions from a classical convolution diagram.

\begin{theorem}[\cite{Z}, Theorem 2.9]\label{Zhao}
   For each $n\geq 1$ the image of the restriction map
    \begin{align*}
        K_{GL_n(\mbC)\times (\mbC^*)^2}(Comm_n) \xrightarrow{i_0^*p_*} K_{GL_n(\mbC)\times (\mbC^*)^2} \simeq \mbZ[q^{\pm},q^{\prime\pm}][z_1^{\pm},\dots,z_n^{\pm}]^{S_n}
    \end{align*}
    is included in the $S_n$-symmetric part of the following ideal 
    \begin{align*}
        \bigcap_{i\neq j\neq k}(1 -  q^{-1} z_j/z_i,1 - q^{\prime-1} z_k/z_j),
    \end{align*}
    where the intersection is taken over all distinct triples $\{i\neq j\neq k\}\subset\{1,\dots,n\}$.
\end{theorem}
The \textit{wheel conditions} are the divisibility conditions on symmetric polynomials lying in the image: if $R\in \mbZ[q^{\pm},q^{\prime\pm}][z_1^{\pm},\dots,z_n^{\pm}]^{S_n}$ lies in the image then
\begin{align*}
R\vert_{z_j-qz_i=0,z_k-q^\prime z_j=0}=0
\end{align*}
 for any triple $\{i\neq j\neq k\}\subset\{1,\dots,n\}$.

The stack $\mathrm{Coh}_n(\mbA^2)$ admits several other useful descriptions of its points:

\begin{itemize}
    \item cotangent stack for the adjoint representation of $GL_n(\mbC)$: the (singular) variety $Comm_n$ is the zero-level $\mu_n^{-1}(0)$ under the moment map $\mu_n: T^*\gl_n\simeq \gl_n^2\to \gl_n^*\simeq\gl_n, ~(x,y)\mapsto [x,y]$ and this way $\mathrm{Coh}_n(\mbA^2)$ is the cotangent stack of $ \gl_n(\mbC)/GL_n(\mbC)$ for the adjoint representation of $GL_n(\mbC)$;
    \item  preprojective stack for a one loop quiver: the stack of finite-dimensional representations the algebra $\mbC[x,y]$=preprojective algebra of a one-loop quiver, defined in (\ref{prep});
    \item it contains on open substack $\mathrm{}{Loc}_n(\Sigma_1)\subset \mathrm{Coh}_n(\mbA^2)$ of $n$-dimensional $\mbC$-representations of the fundamental group $\pi_1(\Sigma_1)=\mbZ^2$ of a genus $1$ Riemann surface.
\end{itemize}

 We adapt the argument in (\ref{Zhao}) to find wheel conditions for representations of reductive groups. 

 \subsection{The statement}

  Let $V$ be a finite dimensional representation of a complex reductive group $G\times T_s$ for some torus $T_s$ leaving invariant the zero-set $\mu_V^{-1}(0)$ under the $G$-equivariant moment map $\mu_V: T^*V \to \g^*$. Suppose the fixed locus $\mu_V^{-1}(0)^{T\times T_s}$ is a point, the origin. Denote its embbedding by $\pt=\mu_V^{-1}(0)^{T\times T_s}\xrightarrow{i_0} V\oplus V^*$, it is lci and the operation $i_0^*$ in K-theory is well-defined. Denote by $W$ the Weyl group of the pair $(G,T)$. Choosing a basis in $V$ we can talk about coordinate lines. Consider two coordinate lines $l\subset V$ and $l^\prime\subset V^*$ and form a commutative diagram of closed embeddings. The lines are $T\times T_s$-invariant, denote by $\chi_l,~\chi_{l^\prime}$ their $T\times T_s$-characters.

\begin{theorem}\label{whTheorem}
    The image under restriction map
    \begin{align*}
        K_{G\times T_s}(\mu_V^{-1}(0))\xrightarrow{i_0^*p_*}K_{G\times T_s}(\pt)
    \end{align*}
    is contained in the ideal 
    \begin{align}
        K_{T\times T_s}(\pt)^W\cap \bigcap_{\Pi} (1 - \chi_{l}^{-1},1 - \chi_{l^\prime}^{-1})
    \end{align}
    where the intersection is taken over the set $\Pi$ of all pairs of coordinate lines $l\subset V,~l^{\prime}\subset V^*$ such that 
    \begin{align*}
        l\oplus l^{\prime}\times_{V\oplus V^*} \mu_V^{-1}(0)=l\cup_{\{0\}}l^{\prime}
    \end{align*}
    scheme-theoretically.
\end{theorem}

\begin{proof}
We first consider the image under restriction map
\begin{align*}
    K_{T\times T_s}(\mu_V^{-1}(0))\xrightarrow{i_0^*p_*}K_{T\times T_s}
\end{align*}
where $T\subset G$ a maximal torus with the Weyl group $W$, and then symmetrize due to an isomorphism $K_{G\times T_s}\simeq K_{T\times T_s}^W, ~ K_{T\times T_s}=\mbZ[X^*(T)\times X^*(T_s)]$.

Consider two coordinate lines $l\subset V$ and $l^\prime\subset V^*$  such $l\oplus l^{\prime}$ intersected with $\mu_V^{-1}(0)$ inside $T^*V$ is $l\cup l^{\prime}$. Then we have a diagram of closed embeddings with Cartesian square
\[\begin{tikzcd}\label{Cart}
	&& \pt \\
	{l\cup l^{\prime}} & {l\oplus l^{\prime}} \\
	{\mu_V^{-1}(0)} & {V\oplus V^*}
	\arrow["{v_0}"', from=1-3, to=2-2]
	\arrow["{i_0}", curve={height=-12pt}, from=1-3, to=3-2]
	\arrow["{p^\prime}", hook, from=2-1, to=2-2]
	\arrow[hook, from=2-1, to=3-1]
	\arrow["{{{i_V}}}", hook', from=2-2, to=3-2]
	\arrow["p", hook, from=3-1, to=3-2]
\end{tikzcd}\]
 
We are interested in the image of $i_0^*p_*=v_0^*i_V^*p_*=v_0^*p^\prime_*i_V^!$ where the last equality is the base change property ($\ref{prsKth}$). Then we have $\im(j^*)\subset \im(v_0^*p^\prime_*)$. The image of $K_{T\times T_s}(l\cup l^{\prime})\xrightarrow{p_{*}^{\prime}} K_{T\times T_s}(l\oplus l^{\prime})$ is generated by $p^\prime_*[\mathcal{O}_l],~p^\prime_*[\mathcal{O}_{l^\prime}]$ as $K_{T\times T_s}(\pt)$-module. To compute the characters $v_0^*p^\prime_*[\mathcal{O}_l],~ v_0^*p^\prime_*[\mathcal{O}_{l^\prime}]$ we use $T\times T_s$-locally free resolutions
\begin{align*}
    0\to \mathcal{O}_{l\oplus l^\prime}(-l)\to \mathcal{O}_{l\oplus l^\prime} \to p^\prime_* \mathcal{O}_l\to 0\\
    0\to \mathcal{O}_{l\oplus l^\prime}(-l^\prime)\to \mathcal{O}_{l\oplus l^\prime} \to p^\prime_* \mathcal{O}_{l^\prime}\to 0.
\end{align*}
Then we have 
\begin{align*}
    v_0^*p^\prime_*[\mathcal{O}_l] = v_0^*[\mathcal{O}_{l\oplus l^\prime}] - v_0^*[\mathcal{O}_{l\oplus l^\prime}(-l)] = 1 - \chi_l^{-1}.
\end{align*}

 Similarly, 
\begin{align*}
    v_0^*p^\prime_*[\mathcal{O}_{l^{\prime}}] = 1 - \chi_{l^\prime}^{-1}.
\end{align*}

We get the image of $i_0^*p_*$ lies inside an ideal 
\begin{align*}
    (1 - \chi_{l}^{-1},1 - \chi_{l^\prime}^{-1}) \subset K_{T\times T_s}(\pt)
\end{align*}
 thus it lies inside the intersection 
 \begin{align*}
     \im(i_0^*p_*)\subset \bigcap_{\Pi} (1 - \chi_{l}^{-1},1 - \chi_{l^\prime}^{-1})
  \end{align*}
over the set $\Pi$ of all pairs of lines $l,~l^{\prime}$ with the required property.

The image of the same map $K_{G\times T_s}(\mu_V^{-1}(0))\xrightarrow{i_0^*p_*} K_{G\times T_s}$ but in $G\times T_s$-equivariant K-theory $K_{G\times T_s}=K_{T\times T_s}^W$ is given by taking $W$-invariants. By above these $W$-invariants sit inside $W$-invariant polynomials from the intersection. QED

 \end{proof}
 We consider two examples:
 adjoint representations of semisimple groups and irreducible representations of $SL_2(\mbC)$.

\subsection{Examples}

$\bullet ~ G\curvearrowright\g^g$

The first and the third descriptions above of $\mathrm{Coh}_n(\mbA^2)$ suggest the following generalization. Let $G$ be a semisimple group and $\g$ be its Lie algebra, considered as adjoint representation. We will use an isomorphism of Lie algebras $\g\simeq\g^*, ~x\mapsto \Tr(\ad_x\circ \ad_-)$, via the Killing form. Let $T\subset G$ be a maximal torus and $W$ be the Weyl group of $(G,T)$. Let $g\geq 1$ be an integer, and consider the \textit{additive character stack}
\begin{align*}
    \mu_{\g}^{-1}(0)/G
\end{align*}
where 
\begin{align*}
    \mu_{\g}^{-1}(0):=\{(a_1,\dots,a_g,b_1,\dots,b_g)\in \g^{2g}: \sum_{i=1}^g[a_i,b_i]=0\}.
\end{align*}
and $G$ acts by component wise conjugation.

Consider the torus 
\begin{align*}
T_s:=\{(q_1,\dots, q_g, q_1^{\prime}, \dots, q_g^{\prime})\in (\mbC^*)^{2g}: q_1q_1^{\prime} = \dots=q_gq_g^{\prime} \} \simeq (\mbC^*)^{g+1}
\end{align*}
acting on $\g^{2g}$ by scaling, and preserving $\mu_{\g}^{-1}(0)$. The fixed locus $\mu_{\g}^{-1}(0)^{T\times T_s}=\{0\}\hookrightarrow\mu_{\g}^{-1}(0)$ is one point.

It is an additive analogue of the character stack, parameterizing $G$ - local systems on a smooth genus $g$ Riemann surface.

Let 
\begin{align*}
\g = \g_0 \oplus \bigoplus_{\alpha \in \Phi} \g_{\alpha}
\end{align*}
be a root decomposition, with $\Phi$ the set of roots. The root spaces $\g_{\alpha}$ are 1-dimensional. Let $h_1,\dots,h_r$ be a basis in the Cartan subalgebra $\g_0$. In this example the set of $\Pi$ from the theorem consists of the following pairs: for each $1\leq i\leq g$:
\begin{itemize}
    \item  a pair $\g_\alpha^{(i)},~\g_\beta^{(i)}$ such that $\alpha+\beta$ is again a root,
    \item a pair $\g_{\alpha}^{(i)},~\g_{-\alpha}^{(i)}$ for each root $\alpha$,
    \item a pair $\mbC h_k^{(i)},\g^{(i)}_{\alpha}$ for each $1\leq k\leq r$ and a root $\alpha$ such that $\alpha(h_k^{(i)})\neq 0$,
    \item similarly, a pair $\g^{(i)}_{\alpha},\mbC h_k^{(i)}$
    
\end{itemize}

These pairs give the following ideals in $K_{T\times T_s}(\pt)$, respectively:
\begin{itemize}
    \item $(1-q_i^{-1}e^{-\alpha},1-q_i^{\prime-1}e^{-\beta})$,
    \item $(1-q_i^{-1}e^{-\alpha},1-q_i^{\prime-1}e^{\alpha})$,
    \item $(1-q_i^{-1},1-q_i^{\prime-1}e^{-\alpha})$,
    \item $(1-q_i^{-1}e^{-\alpha},1-q_i^{\prime-1})$.
\end{itemize}

\begin{theorem}
        
    The image of restriction map 
    \begin{align*}
        K_{G\times T_s}(\mu_{\g}^{-1}(0)) \xrightarrow{i_0^*p_*} K_{G\times T_s}(\pt)
    \end{align*}
    is included into the ideal
    \begin{align*}
K_{T\times T_s}(\pt)^W \cap
\bigcap_{i=1}^{g}
\bigg[
&
\bigcap_{\substack{\alpha,\beta\in\Phi\\ \alpha+\beta\in\Phi}}
\left(
1-q_i^{-1}e^{-\alpha},
1-q_i^{\prime -1}e^{-\beta}
\right)
\\
&\cap
\bigcap_{\alpha\in\Phi}
\left(
1-q_i^{-1}e^{-\alpha},
1-q_i^{\prime -1}e^{\alpha}
\right)
\\
&\cap
\bigcap_{\substack{\alpha\in\Phi,\; h_k\\ \alpha(h_k)\neq 0}}
\left(
1-q_i^{-1},
1-q_i^{\prime -1}e^{-\alpha}
\right)
\\
&\cap
\bigcap_{\substack{\alpha\in\Phi,\; h_k\\ \alpha(h_k)\neq 0}}
\left(
1-q_i^{-1}e^{-\alpha},
1-q_i^{\prime -1}
\right)
\bigg].
\end{align*}
        
\end{theorem}

The same trick works for adjoint representation $GL_n(\mbC)\curvearrowright\gl_n(\mbC)$, where we identify $\gl_n(\mbC)\simeq \gl_n(\mbC)^*$ via trace form. 

\vspace{1cm}

Consider $Sp_4(\mbC)\curvearrowright\mathfrak{sp_4}(\mbC)$. Denote by $\chi_1,~\chi_2$ the fundamental weights of the adjoint representation.
It has 8 roots:
\[\begin{tikzcd}
	&& {2\chi_2} \\
	& {\chi_2-\chi_1} && {\chi_1+\chi_2} \\
	{-2\chi_1} && {.} && {2\chi_1} \\
	& {-\chi_1-\chi_2} && {\chi_1-\chi_2} \\
	&& {-2\chi_2}
	\arrow[from=3-3, to=1-3]
	\arrow[from=3-3, to=2-2]
	\arrow[from=3-3, to=2-4]
	\arrow[from=3-3, to=3-1]
	\arrow[from=3-3, to=3-5]
	\arrow[from=3-3, to=4-2]
	\arrow[from=3-3, to=4-4]
	\arrow[from=3-3, to=5-3]
\end{tikzcd}\]

Here the set $\Pi$ consists of 
We consider all the pairs of roots $(\alpha, \beta)$ such that $\alpha+\beta$ is again a root. We have 24 of such:
\begin{align*}
    \pm(\chi_1 - \chi_2, \chi_1+\chi_2)
    &\qquad
    \pm(\chi_1+\chi_2, \chi_1 - \chi_2),\\
    \pm(-2\chi_1, \chi_1+\chi_2)
    &\qquad
    \pm(\chi_1+\chi_2, -2\chi_1),\\
    \pm(-2\chi_2, \chi_1+\chi_2)
    &\qquad
    \pm(\chi_1+\chi_2, -2\chi_2),\\
    \pm(-\chi_1 + \chi_2, \chi_1+\chi_2)
    &\qquad
    \pm(\chi_1+\chi_2, -\chi_1 + \chi_2),\\
    \pm(2\chi_2, \chi_1-\chi_2)
    &\qquad
    \pm(\chi_1-\chi_2, 2\chi_2),\\
    \pm(-2\chi_1, \chi_1-\chi_2)
    &\qquad
    \pm(\chi_1-\chi_2, -2\chi_1).
\end{align*}

Also we have root-antiroot pairs:
\begin{align*}
    (\chi_1+\chi_2, -\chi_1-\chi_2) & \qquad (-\chi_1 -\chi_2,\chi_1+\chi_2)\\
    (2\chi_1, -2\chi_1) &\qquad (-2\chi_1,2\chi_1)\\
    (\chi_1 - \chi_2, -\chi_1 +\chi_2) &\qquad (-\chi_1 +\chi_2, \chi_1 -\chi_2)\\
    (2\chi_2,-2\chi_2) & \qquad (-2\chi_2, 2\chi_2).
\end{align*}

Choose a basis $h_1,h_2$ in $\g_0$ such that $\chi_{i}(h_j)=\delta_{ij}$. Then we have Cartan-root pairs:

\[
\begin{aligned}
&\left\{
(\mathbb C h_1,\mathfrak g_\alpha)
:
\alpha\in
\{\pm 2\chi_1,\pm(\chi_1+\chi_2),\pm(\chi_1-\chi_2)\}
\right\}
\\
&\cup
\left\{
(\mathbb C h_2,\mathfrak g_\alpha
:
\alpha\in
\{\pm 2\chi_2,\pm(\chi_1+\chi_2),\pm(\chi_1-\chi_2)\}
\right\}
\\
&\cup
\left\{
(\mathfrak g_\alpha,\mathbb C h_1
:
\alpha\in
\{\pm 2\chi_1,\pm(\chi_1+\chi_2),\pm(\chi_1-\chi_2)\}
\right\}
\\
&\cup
\left\{
(\mathfrak g_\alpha,\mathbb C h_2
:
\alpha\in
\{\pm 2\chi_2,\pm(\chi_1+\chi_2),\pm(\chi_1-\chi_2)\}
\right\}.
\end{aligned}
\]

$\bullet ~SL_2(\mbC)\curvearrowright Sym^n(\mbC)$

Let $G=SL_2(\mbC)$ and $V=Sym^n(\mbC^2)$ be its irreducible representation of dimension $n+1$. Denote by $T=\begin{pmatrix}
    z& 0\\ 0& z^{-1}
\end{pmatrix}\subset SL_2(\mbC)$ the maximal torus. The action of $g\in SL_2(\mbC)$ on the dual $V^*$ is defined by $g.\xi^* (v) = \xi^*(g^{-1}v)$. Let 
\begin{align*}
 T_s = \mbC^*_q   
\end{align*}
 acts on $V$ via the induced action on $\mbC^2$ 
 \begin{align*}
     q.v = qv,~v\in V \qquad q.\xi = q^{-1}\xi,~\xi\in V^*.
 \end{align*}
 This action commutes with the action of $SL_2(\mbC)$. We have 

\begin{theorem}
    The image of the restriction map 
    \begin{align*}
        K_{SL_2(\mbC)\times T_s}(\mu_{Sym^n(\mbC^2)}^{-1}(0)) \to K_{SL_2(\mbC)\times T_s}(\pt)
    \end{align*}
    is included into the ideal 
    \begin{align*}
K_{T\times T_s}(\pt)^{S_2}\cap
&
\bigcap_{\substack{k+l=n\\ l\geq 1}}
\left(
1-z^{-(k-l)}q^{-n}, 1-z^{k-l+2}q^{n}
\right)
\\
&\cap
\bigcap_{\substack{k+l=n\\ k\geq 1}}
\left(
1-z^{-(k-l)}q^{-n}, 1-z^{k-l-2}q^n
\right)
\\
&\cap
\bigcap_{\substack{k+l=n\\ k,l\geq 0\\ k\neq l}}
\left(
1-z^{-(k-l)}q^{-n}, 1-z^{k-l}q^n
\right).
\end{align*}

    where $S_2$ acts on monomials by $z^aq\mapsto z^{-a}q$
\end{theorem}

\begin{proof}

 Choose a basis in $V$ 
\begin{align*}
    V = \mbC\langle e_1^ke_2^l:k+l=n \rangle,
\end{align*}
where $e_1,e_2$ is a basis in $\mbC^2$. Denote by $e_1^*,e_2^*$ the dual basis, and by $(e_1^*)^k(e_2^*)^l$ the dual basis in the dual vector space $V^*$.
Choose a standard basis $e,f,h$ in $\mathfrak{sl}_2(\mbC)$ in which the Lie bracket is $[h,e] = 2e,[h,f] = -2f, [e,f] = h$.

The action of $\mathfrak{sl}_2$ on basis vectors in $\mbC^2$ and its dual is 
\begin{align*}
    ee_1 = 0, fe_2 = 0, &&  ee_1^* = 0, fe_2^* = 0,\\
    ee_2 = e_1, fe_1 = e_2, && ee_2^* = -e_1^*, fe_1^* = -e_2^*\\
    he_1 = e_1, he_2 = -e_2 && he_1^* = -e_1^*, he_2^* = e_2^*. 
\end{align*}
Action on basis vectors of
$V$ and $V^*$ is defined by 
\begin{align*}
    e(e_1^ke_2^l) = le_1^{k+1}e_2^{l-1}, && e((e_1^*)^k(e_2^*)^l) = -l(e_1^*)^{k+1}(e_2^*)^{l-1}\\
    f(e_1^{k}e_2^l) = ke_1^{k-1}e_2^{l+1}, && f((e_1^*)^k(e_2^*)^l) = -k(e_1^*)^{k-1}(e_2^*)^{l+1}\\
    h(e_1^ke_2^l) = (k-l)e_1^ke_2^l, && h((e_1^*)^k(e_2^*)^l) = (-k+l)(e_1^*)^k(e_2^*)^l
\end{align*}
We compute the pairings 
\begin{align*}
    \langle (e_1^*)^a(e_2^*)^b, e.(e_1^{k}e_2^{l})\rangle = l\delta_{a,k+1}\delta_{b,l-1}\\
    \langle (e_1^*)^a(e_2^*)^b, f.(e_1^{k}e_2^{l})\rangle = k\delta_{a,k-1}\delta_{b,l+1}\\
    \langle (e_1^*)^a(e_2^*)^b, h.(e_1^{k}e_2^{l})\rangle = (k-l)\delta_{a,k}\delta_{b,l}.
\end{align*}

The $T\times T_s$-characters of 
\begin{align*}
 e_1^ke_2^l,  \qquad (e_1^*)^{k}(e_
2^*)^l \end{align*} are
\begin{align*}
    z^{k-l}q^{k+l}, \qquad z^{-(k-l)}q^{-k-l}.
\end{align*}

The zero level of the moment map consists of pairs 
\begin{align*}
    \mu_V^{-1}(0) = \{(x,x^*)\in V\oplus V^*: \langle x^*, a.x\rangle = 0~ \forall a\in \mathfrak{sl}_2(\mbC)\}
\end{align*}

Consider the fiber diagram
\[\begin{tikzcd}
	& {\mbC e_1^ke_2^l\oplus \mbC(e_1^*)^a(e_2^*)^b} \\
	{\mu_V^{-1}(0)} & {V\oplus V^*}
	\arrow["{{{i_V}}}", hook', from=1-2, to=2-2]
	\arrow["p", hook, from=2-1, to=2-2]
\end{tikzcd}\]

We would like to impose conditions on $(a,b,k,l), k+l = n$ such that the fiber product be the union of two coordinate lines, intersecting along the origin. The set $\Pi$ of admissible lines is:   

\begin{align*}
\mbC e_1^ke_2^l \cup_{\{0\}} \mbC(e_1^*)^{k+1}(e_2^*)^{l-1},~ l\geq 1 &&   \mbC e_1^ke_2^l\cup_{\{0\}} \mbC(e_1^*)^{k-1}(e_2^*)^{l+1},~ k\geq 1  
\end{align*}
\begin{align*}
    \mbC e_1^ke_2^l \cup_{\{0\}} \mbC(e_1^*)^{k}(e_2^*)^{l}, ~k\neq l
\end{align*}

and the corresponding ideals are:
\begin{align*}
    (1-z^{-(k-l)}q^{-n}, 1-z^{k-l+2}q^{n})&&(1-z^{-(k-l)}q^{-n}, 1-z^{k-l-2}q^n)
\end{align*}
\begin{align*}
    (1-z^{-(k-l)}q^{-n}, 1-z^{k-l}q^n)
\end{align*}

We are done. 

\end{proof}

\end{document}